\documentclass[11pt, twoside]{amsart}

\title[Monoidal adjunctions and abelian envelopes]{Monoidal adjunctions and abelian envelopes}

\author{Johannes Flake}
\address{Mathematical Institute, University of Bonn,
Endenicher Allee 60, 53115 Bonn, Germany}
\email{flake@math.uni-bonn.de}

\author{Robert Laugwitz}
\address{School of Mathematical Sciences,
University of Nottingham, University Park, Nottingham, NG7 2RD, UK}
\email{robert.laugwitz@nottingham.ac.uk}

\author{Sebastian Posur}
\address{University of Münster,
Fachbereich Mathematik und Informatik,
Einsteinstraße 62,
48149 Münster,
Germany}
\email{sebastian.posur@uni-muenster.de}

\usepackage[english]{babel}
\usepackage{import}
\usepackage{amsmath,amsfonts,amsthm,amssymb}

\usepackage[alphabetic,abbrev,nobysame]{amsrefs} 
\renewcommand\MR[1]{}

\usepackage{url}
\usepackage{fancyhdr}
\usepackage{graphicx}
\usepackage{microtype}
\usepackage{wrapfig, caption}
\usepackage[
colorlinks=true,
linkcolor=black, %
anchorcolor=black,%
citecolor=black, %
urlcolor=black, %
]{hyperref}
\usepackage{zref-clever} \newcommand\Cref[1]{\zcref[S]{#1}}
\usepackage{xcolor}
\usepackage{geometry}
\usepackage{tikz}
\usetikzlibrary{cd,calc}
\usepackage[all,cmtip]{xy}
\usepackage{dutchcal} %
\usetikzlibrary{arrows}

\newcommand{\longrightleftarrows}{\resizebox{18pt}{7pt}{$\rightleftarrows$}}

\newcommand{\adj}[4]{#1\colon #2~\longrightleftarrows~ #3\colon #4}

\newcommand\SL{\mathsf{SL}}

\newcommand{\rmod}[1]{\mathbf{mod}\text{-}#1}

\newcommand{\Aut}{\operatorname{Aut}}
\newcommand\coev{\operatorname{coev}}

\newcommand{\ev}{\operatorname{ev}}
\newcommand{\End}{\operatorname{End}}

\newcommand{\Hom}{\operatorname{Hom}}
\newcommand{\one}{\mathbf{1}}

\newcommand{\Rep}{\operatorname{Rep}}

\newcommand{\uRep}{\underline{\operatorname{Rep}}\,}
\newcommand{\uPerm}{\underline{\operatorname{Perm}}}
\newcommand\Ten{\operatorname{Ten}}

\newcommand{\cA}{\mathcal{A}}
\newcommand{\cB}{\mathcal{B}}
\newcommand{\cC}{\mathcal{C}}
\newcommand{\cD}{\mathcal{D}}
\newcommand{\cI}{\mathcal{I}}
\newcommand{\cP}{\mathcal{P}}
\newcommand{\cT}{\mathcal{T}}
\newcommand{\cU}{\mathcal{U}}
\newcommand{\cZ}{\mathcal{Z}}

\newcommand{\cH}{\mathcal{H}}

\newcommand{\cS}{\mathcal{S}}

\newcommand\RepSt{\uRep \mathsf{S}_t}
\newcommand\RepOt{\uRep \mathsf{O}_t}
\newcommand\ZZ{\mathbb{Z}} %
\newcommand\kk{\Bbbk} %

\newcommand\id{{\operatorname{id}}}

\newcommand\s\sigma
\newcommand\p\pi

\renewcommand\o\otimes

\newcommand\h{V_n}

\newcommand\im{\operatorname{im}}

\newcommand\cSet{\mathcal{Set}}
\newcommand\cAb{\mathcal{Ab}}
\newcommand\PSh{\operatorname{PSh}}
\newcommand\Ind{\operatorname{Ind}}

\renewcommand\l\lambda

\renewcommand\o{\otimes}
\newcommand\x\times
\newcommand\tto\twoheadrightarrow

\newcommand\NN{\mathbb{N}}

\numberwithin{equation}{section}
\newcommand\newtheoremx[3]{%
\AddToHook{env/#1/begin}{
\zcsetup{countertype={#2=#1}}}
\newtheorem{#1}[#2]{#3}
}

\newtheorem{introtheorem}{Theorem}

\zcsetup{countertype={introtheorem=theorem}}

\newtheorem{theorem}{Theorem}[section]
\newtheoremx{claim}{theorem}{Claim}
\newtheoremx{proposition}{theorem}{Proposition}
\newtheoremx{corollary}{theorem}{Corollary}
\newtheoremx{lemma}{theorem}{Lemma}
\newtheoremx{conjecture}{theorem}{Conjecture}
\newtheorem*{theorem*}{Theorem}
\newtheoremx{assumption}{theorem}{Assumption}

\theoremstyle{definition}
\newtheoremx{definition}{theorem}{Definition}

\theoremstyle{remark}
\newtheoremx{example}{theorem}{Example}
\newtheoremx{remark}{theorem}{Remark}
\newtheoremx{question}{theorem}{Question}

\renewcommand\:{\colon}

\newcommand\Cps{\cC^{\text{ps}}}
\newcommand\Cob{\mathcal{Cob}}

\newcommand\RepHt{\uRep\mathsf{H}_t}
\newcommand\RepGLnFq{\uRep\mathsf{GL}_t(\F_q)}

\newcommand\Cb{\mathrm{Cob}}

\newcommand\F{\mathbb{F}}

\newcommand\U{\mathcal{U}}
\newcommand\Uex{\U^{\operatorname{ex}}}

\newcommand\tforall{\text{for all }}

\renewcommand\hat\widehat
\renewcommand\tilde\widetilde
\newcommand\op{^{\text{op}}}
\newcommand\ti[1]{\widetilde{#1}}
\newcommand\colim{\operatorname{colim}}

\newcommand\mymatrix[1]{\left(\begin{smallmatrix}#1\end{smallmatrix}\right)}

\tikzset{
    partition/.style={
      scale=0.4,
      yscale=-1,
      baseline={([yshift=-0.5ex]current bounding box.center)}
    }
}

\tikzset{
    bend/.cd,
    0/.style={},
    1/.style={bend right},
    -1/.style={bend left}
}
\newcommand\makePartPt[1]{({Mod(#1,10)},{(#1-Mod(#1,10))*.1})}
\newcommand\makePartLn[2]{%
\pgfmathtruncatemacro\bend{%
(int(#1/10)==int(#2/10)) ?
(#1<10 ? 1 : -1)*(#1>#2 ? 1 : -1)
: 0%
}
\draw[draw=black,line width=0.5pt,line cap=round] ({mod(#1,10)},{int(#1/10}) to[bend/\bend] ({mod(#2,10)},{int(#2/10)});
}
\newcommand\tp[1] {%
~\tikz[partition] {
\draw[white,opacity=0] (1,0)--(1,1); %
\def\j{0}
\foreach \i [remember=\i as \j] in {#1} {
  \ifnum \i>0
    \ifnum \j>0
      \makePartLn{\i}{\j};
    \fi
  \fi
} %
\foreach \i [remember=\i as \j] in {#1} {
  \ifnum \i>0
    \draw[draw=none,fill=black] \makePartPt\i circle (4pt);
  \fi
} %
}~}

\newcommand\listorempty[1]{\def\temp{#1}\ifx\temp\empty \emptyset \else (#1) \fi}

\renewcommand\bar{\tp{1,11}}

\usepackage{tikz}

\newcommand\jtqftpt[1]{
\pgfmathsetmacro\x{mod(#1,10)}
\pgfmathsetmacro\y{1-2*div(#1,10)}
\coordinate (n-#1-1) at (\x*\w+\r*\y*\w,\y*\h);
\coordinate (n-#1-2) at (\x*\w-\r*\y*\w,\y*\h);
\draw (n-#1-2) to[out=90,in=90] (n-#1-1) (n-#1-2);
\pgfmathsetmacro\sign{#1 < 10 ? 1 : -1}
\pgfmathtruncatemacro\txt{#1 < 10 ? #1 : (#1-10)}
\pgfmathsetmacro\suff{#1 < 10 ? "" : "'"}
\node at (\x*\w,\y*\h+\sign*\doff) {\txt\suff};
}

\newcommand\jtqfthandle[2]{
\bgroup
\def\x{#1*\w}\def\y{#2*\h}\def\ww{0.33cm}\def\hh{0.12cm}\def\ang{80}
\def\arcA{(\x,\y) to[out=\ang,in=180-\ang] (\x+\ww,\y)}
\def\arcB{(\x,\y+\hh) to[out=-\ang,in=-180+\ang] (\x+\ww,\y+\hh) -- (\x+.5*\ww,\y+3*\hh) -- cycle}
\begin{scope}[looseness=1.1]
\begin{scope}
\clip \arcA; \draw[fill=white] \arcB; 
\end{scope}
\draw \arcA;
\end{scope}
\egroup
}

\newcommand\jtqft[2]{ \bgroup
\def\w{0.8cm} \def\h{0.6cm}
\def\r{0.33} \def\doff{0.35cm} 
\def\loocirc{0.6} \def\looshape{0.9}
\tikz[remember picture,looseness=\loocirc,line width=0.8pt,font={\sffamily\small},baseline=0]{
\foreach \i [remember=\i as \ii (initially 0)] in {#1,0} {
\ifnum\i>0 \jtqftpt{\i}; \fi
\ifnum\ii=0
   \global\edef\jtqftstart{\i}
   \global\edef\jtqftpath{(n-\i-1)}
\else\ifnum\i=0
   \pgfmathsetmacro\dout{\ii<10 ? -90 : 90}
   \pgfmathsetmacro\din{\jtqftstart<10 ? -90 : 90}
   \pgfmathsetmacro\dloo{\ii==\jtqftstart ? 3*\looshape :((\ii<10) == (\i<10) ? 1.3*\looshape : \looshape)}
   \global\edef\jtqftpath{\jtqftpath to[out=\dout,in=\din,looseness=\dloo] (n-\jtqftstart-2) to[out=-90,in=-90] (n-\jtqftstart-1)}
   \draw[fill=white,fill opacity=0.9] \jtqftpath;
\else
   \pgfmathsetmacro\dout{\ii<10 ? -90 : 90}
   \pgfmathsetmacro\din{\i<10 ? -90 : 90}
   \pgfmathsetmacro\dloo{(\ii<10) == (\i<10) ? 1.3*\looshape : \looshape}
   \global\edef\jtqftpath{\jtqftpath to[out=\dout,in=\din,looseness=\dloo] (n-\i-2) to[out=-90,in=-90] (n-\i-1)}
\fi\fi } %
#2
} %
\egroup }

\begin{document}

\begin{abstract}
We show how monoidal adjunctions can be used to prove the existence of monoidal abelian envelopes of pseudo-tensor categories, in particular, those admitting a combinatorial description with certain properties.
We derive concrete general criteria, which we then demonstrate by giving relatively simple combinatorial proofs of the existence of new abelian envelopes for interpolation categories of the hyperoctahedral and of the modified symmetric groups.
\end{abstract}

\maketitle

\section{Introduction}

A key technique in modern representation theory is to consider the entirety or at least a sufficiently large collection of representations of an algebraic object, like a group, and study the category they form. In some cases, such categories can be described efficiently by a suitable graphical calculus, like that of Temperley--Lieb diagrams in the case of the group $\SL_2$. Many categories defined using a similar diagrammatic calculus have been studied. Typically, these are linear monoidal categories which can be completed to so-called pseudo-tensor categories. While the latter have a good notion of tensor product and direct sum, it generally does not make sense to consider composition series for their objects. Another completion process called monoidal abelian envelope, or just abelian envelope, addresses this. 

A general theory of abelian envelopes is currently under construction, see \cites{Cou-monab, BEO, CEOP}. However, at this point, there are few easy existence criteria that apply to large classes of pseudo-tensor categories. We suggest a systematic use of monoidal adjunctions, from which we derive, and verify in some cases, combinatorial criteria for the existence of abelian envelopes.

\medskip

We start by defining, in \Cref{sec:pdiag-categories}, a class of linear monoidal categories whose hom-spaces are spanned by collections of morphisms closed under tensor products. As this is a typical feature in categories defined using a diagrammatic calculus, we use the term \emph{pseudo-diagrammatic} for such categories. In \cite{CEOP}, a necessary exactness condition for the existence of abelian envelopes is established. We show:

\begin{introtheorem}[\Cref{prop:U-equals-Uex}] \label{thm:A} Any pseudo-diagrammatic linear monoidal category $\cC$ satisfies the necessary exactness condition of \cite{CEOP} to have an abelian envelope.
\end{introtheorem}

We verify, in \Cref{sec:families}, that \Cref{thm:A} applies, in particular, to various families of interpolation categories based on cobordisms, partitions, or (linear) relations, including $\RepSt$ \cites{Del,CO-ab}, $\RepHt$ \cite{FM}, $\RepOt$ \cite{Del}, Knop's tensor envelopes \cite{Knop}, including Knop's $\RepGLnFq$ \cite{EAH-glt}, and Harman--Snowden's tensor categories \cite{HS} from pro-oligomorphic groups.

By the theory developed in \cite{CEOP}, the existence of an abelian envelope can then be deduced from the existence of \emph{splitting objects}, i.e., objects tensoring with which turns any fixed morphism into a split morphism (see \Cref{sec:ab-env-diag}). The abelian envelopes obtained in this way have the so-called \emph{quotient property}, that is, any object of the abelian envelope is a quotient of one in the original category. Abelian envelopes with the quotient property generalize abelian envelopes with enough projectives, which were constructed in \cite{BEO}.

A key insight is that splitting objects can be transferred along certain monoidal adjunctions (\Cref{cor:transport-splitting}), which implies:

\begin{introtheorem}[\Cref{cor:ab-env-from-functor-semisimple}, \Cref{cor:criterion-enough-proj}] \label{thm:B} Assume $\cC$ is a pseudo-diagrammatic pseudo-tensor category $\cC$ which has a linear monoidal functor $F$ to an abelian tensor category with enough projectives such that $F$ has a left or right adjoint. Then $\cC$ has an abelian envelope with enough projectives. 
\end{introtheorem}

For instance, the target of $F$ could be semisimple. We also show versions of \Cref{thm:B} when some of the assumptions are satisfied only for a family of full subcategories $(\cC_i)_i$ of $\cC$ such that $\cC=\bigcup_i \cC_i$, or when the adjoint functors exist only after passing to ind-completions. In the latter case, the abelian envelopes might not have enough projectives, but still has the quotient property.

To obtain practical criteria for checking the hypotheses in \Cref{thm:B}, we show that for certain linear monoidal functors $F$, the existence of a right adjoint (for instance) already follows from representability of the functor $\Hom(F(-),\one)$, see \Cref{corollary:adjoints_rigid_case} and \Cref{lem:adjoint-ind}.

To describe concrete realizations of abelian envelopes with enough projectives, we slightly generalize and simplify the construction in \cite{BEO}, and show:

\begin{introtheorem}[\Cref{thm:ab-env-projectives}] \label{thm:x} For a pseudo-tensor category $\cC$, the following are equivalent:
\begin{enumerate}
\item[(a)] $\cC$ has an abelian envelope with enough projectives.

\item[(b)] $\one\in\cC$ has a presentation by global splitting objects in $\cC$.

\item[(c)] $\cC$ has a non-zero global splitting object and the exactness condition of \cite{CEOP} holds.
\end{enumerate}
In this case, the abelian envelope is realized by the category of finitely presented functors of the category of splitting objects in $\cC$.
\end{introtheorem}

Here, a \emph{global} splitting object is one that splits all morphisms in $\cC$. In particular, if $\cC$ is assumed pseudo-diagrammatic, then it has an abelian envelope with enough projectives if and only if it has a non-zero global splitting object, by a combination of \Cref{thm:x} and \Cref{thm:A}.

As an application of our general theory, we consider subcategories of Deligne's interpolation categories $\cS_t:=\RepSt$ for the symmetric groups, where the role of the functor $F$ in \Cref{thm:B} is played by the natural inclusion functor to $\cS_t$.

\begin{introtheorem}[\Cref{thm:partition-cats}] \label{thm:C} In characteristic $0$, let $\cC$ be a subcategory of $\cS_t$ whose hom-spaces are spanned by partition diagrams such that the functor $\cS_t(-,\one)|_{\cC_i}$ is representable, for all $i\ge0$, for the pseudo-abelian subcategories $\cC_i$ of $\cC$ generated by the objects $[0],[1],\dots,[i]$. Then $\cC$ has an abelian envelope with the quotient property.
\end{introtheorem}

If even the functor $\cS_t(-,\one)|_{\cC}$ is representable, then the abelian envelope has enough projectives.

Even more specifically, we consider interpolation categories for the hyperoctahedral groups $S_2\wr S_n$ and for the modified symmetric groups $S_2\times S_n$ that are suitable subcategories of the categories $\cS_t$. Verifying the hypothesis of \Cref{thm:C} we show:

\begin{introtheorem}[\Cref{cor:main-H}, \Cref{cor:main-S-prime}]  \label{thm:E} The interpolation categories $(\cH_t)_{t\in\kk}$ and $(\cS'_t)_{t\in\kk\setminus\{0\}}$ of the hyperoctahedreal and the modified symmetric groups, respectively, have abelian envelopes with the quotient property, over any field of characteristic~$0$. The abelian envelopes of $(\cS'_t)_{t\in\kk\setminus\{0\}}$ have enough projectives.
\end{introtheorem}

In both cases, the existence result seems to be new. An advantage of our proof strategy for the existence of abelian envelopes, as demonstrated in the two examples, is that it is relatively simple and combinatorial in nature, and that it is applicable and might lead to a uniform theory of abelian envelopes for large classes of (interpolation) pseudo-tensor categories. 
The criterion in \Cref{thm:C} also gives a new proof for the well-known fact due to \cite{CO-ab} that the interpolation categories $(\cS_t)_{t\in\kk}$ have abelian envelopes in characteristic~$0$, see \Cref{cor:St}. 

Our results in the concrete cases can also be interpreted as follows: for any pseudo-tensor category $\cC$ with an abelian envelope $\widetilde\cC$ and with a faithful linear monoidal functor to $\cS_t$, for some $t\in\kk$, note that the universal property of the abelian envelope implies the existence of an (exact) tensor functor from $\widetilde\cC$ to the abelian envelope $\widetilde\cS_t$ of $\cS_t$ extending the original functor. Hence, by a suitably general version of Tannaka--Krein duality (\cite{CEO-incompressible}*{Theorem~4.2.1}), $\widetilde\cC$ can be viewed as the category of representations of an $\widetilde\cS_t$-group, i.e., an affine group scheme in $\widetilde\cS_t$ together with a morphism from the fundamental group of $\widetilde\cS_t$. In particular, the abelian envelopes constructed in \Cref{thm:E} correspond to such $\widetilde\cS_t$-groups.

Crucial ingredients in the proof of \Cref{thm:C} are certain restriction functors for the categories $\cS_t$. In a follow-up paper \cite{FLP-future}, we study more general restriction and induction functors between categories constructed using partitions and cobordisms, and use the theory developed in this paper to connect such functors with abelian envelopes. In particular, we show that the abelian envelopes of the categories $\cH_t$ (also) have enough projectives.

\subsection*{Acknowledgments} J.~F.~thanks Kevin Coulembier and Nate Harman for useful discussions, and the Max Planck Institute for Mathematics
Bonn for the excellent conditions provided while some of this article was written.

\tableofcontents

\section{Background}

We fix a field $\kk$.

\subsection{Pseudo-tensor categories} \label{bg:pseudo}
We recall basic notions relating to various kinds of categories following \cites{EGNO,CEOP}.

A monoidal category is called \emph{left rigid / right rigid / rigid} if it has left / right / both left and right duals. Our convention, as in \cite{EGNO}, is to denote the left dual of an object $X$ by $X^*$ and its right dual by ${}^*X$, i.e., we have a coevaluation $\one\to X\o X^*$ and an evaluation $X^*\o X\to\one$ in this case. A rigid object $X$ comes with two pairs of adjoint functors $(-\o X,-\o X^*)$ and $(X\o-,{}^*X\o-)$, and similar adjunctions can be formulated for objects that are only left or right rigid. 

A $\kk$-linear additive category is called \emph{pseudo-abelian} if idempotent endomorphisms split. Any abelian $\kk$-linear category is pseudo-abelian. A $\kk$-linear additive category is called \emph{semisimple} if it is abelian and any object is isomorphic to a direct sum of finitely many simple objects. For an equivalent definition of semisimple categories not using the notion of abelian categories, see \cite{Muger}*{Section~2.1}.

From any $\kk$-linear category $\cC$, one can construct a pseudo-abelian category by adding formal direct sums and formal images of idempotents. We call this category the \emph{pseudo-abelian envelope} and denote it $\Cps$. If $\cC$ is rigid \slash{} symmetric \slash{} monoidal, then the respective structure can be extended naturally to $\Cps$. We define the \emph{tensor product} of two pseudo-abelian categories $\cC$ and $\cD$ as $\cC\boxtimes \cD:=(C\times_\Bbbk D)^{\mathrm{ps}}$, for the naive tensor product  $C\times_\Bbbk D$, see also \cite{FLP}*{Section 2.1}.

Following \cite{CEOP}*{Section~1.1.1}, for us, a \emph{pseudo-tensor category} is a $\kk$-linear pseudo-abelian rigid monoidal category which additionally satisfies the following conditions:
\begin{itemize}
    \item it is essentially small and
    \item the endomorphism algebra of the tensor unit is one-dimensional.
\end{itemize}
A pseudo-tensor category is called a \emph{tensor category} (here, as in \cite{CEOP}\footnote{There is a typo in \cite{CEOP}*{1.1.4}: The expression ``pseudo-abelian category'' in the first sentence of the subsection should be ``pseudo-tensor category''.}) if it is abelian (note that the definition of tensor categories in \cite{EGNO} is stronger). Since every object has a left and right dual by assumption, tensoring any object has both a left and right adjoint functor, and hence, the tensor product is biexact. 

For pseudo-tensor categories $\cD,\cD'$, we denote by $\Ten^{\text{faith}}(\cD,\cD')$ the category of faithful linear monoidal functors. For tensor categories $\cT,\cT'$, we denote by $\Ten^{\text{exact}}(\cT,\cT')$ the category of exact linear monoidal functors (these are exactly the faithful linear monoidal functors by \cite{CEOP}*{Theorem~2.4.1}). 

\begin{definition}[\cite{CEOP}*{Definition~2.1.1}] A \emph{monoidal abelian envelope}, or just \emph{abelian envelope}, of a pseudo-tensor category $\cD$ is a pair $(\cT,F)$, where $\cT$ is a tensor category and $F\in\Ten^{\text{faith}}(\cD,\cT)$ is a fully faithful functor such that for any tensor category $\cT'$, the functor
$$
(-\circ F): \Ten^{\text{exact}}(\cT,\cT')\to\Ten^{\text{faith}}(\cD,\cT')
$$
is an equivalence.
\end{definition}

We say a monoidal / $\kk$-linear additive monoidal / pseudo-abelian monoidal category $\cC$ has a \emph{generator} $X$, where $X$ is an object in $\cC$, if $\cC$ coincides with any a monoidal,  $\kk$-linear additive monoidal or pseudo-abelian full subcategory containing $X$, respectively.

The \emph{ind-completion} of a pseudo-tensor category $\cC$ is denoted $\Ind(\cC)$. It is the completion of $\cC$ with respect to taking filtered colimits. Concretely, it can be thought of as the ``category of filtered diagrams'' in $\cC$, that is, a category whose objects are represented by functors from small filtered categories to $\cC$. Post-composition of functors yields a functor $\Ind(F)\:\Ind(\cC)\to\Ind(\cD)$ for every functor $F\:\cC\to\cD$. We refer to \Cref{appendix:ind} for a detailed discussion of ind-completions and some of their properties that we will require.

\subsection{Categories defined using cobordisms} \label{bg:Cob}

We recall two equivalent descriptions of the categories $\Cob_{\alpha,u}$, one in terms of objects and morphisms and one in terms of a universal property, following \cite{KS} and \cite{KOK}, see also \cite{FLP}.

For any $m,m'\ge0$, let $\Cb_{m,m'}$ be the set of isotopy classes of compact orientable smooth 2d surfaces $M$ whose boundary $\partial M$ is a disjoint union of $m+m'$ circles labeled by the set $\{1,\dots,m,1',\dots,m'\}$, where isotopies have to fix the boundary. We call the elements of these sets \emph{cobordisms}. The circles labeled by $\{1,\dots,m\}$ are called \emph{upper circles}, and the remaining circles \emph{lower circles}.
Let $\cC$ be the category with objects $\{[m]\}_{m\ge0}$ and with morphisms $\cC([m],[m'])=\kk\Cb_{m,m'}$, where gluing along boundary circles with the same labels defines the composition. The identity morphisms in this category are given by disjoint unions of cylinders. Disjoint union and relabeling the boundary circles (preserving the order) defines a functor $\o\colon \cC\times \cC\to \cC$ for this category. With these two operations, $\cC$ becomes a $\kk$-linear monoidal category with tensor unit $[0]$, corresponding to an empty set of boundary circles. Since cobordisms can have arbitrary closed components (without  boundary), and since there are no restrictions on the genus of cobordisms, all hom-spaces in $\cC$ are infinite-dimensional.

For instance, we define
$$
\eta:=\jtqft{11}{}, \qquad
\mu:=\jtqft{1,2,11}{}, \qquad
\varepsilon:=\jtqft{1}{}, \qquad
\Delta:=\jtqft{1,12,11}{}, \qquad
\Psi:=\jtqft{1,12,0,2,11}{},
$$ 
$$
\phi:=\mu\Delta=\jtqft{1,11}{\jtqfthandle{0.8}{0}},\qquad
\hat \phi_i := \varepsilon \phi^i \eta = \jtqft{0}{
\draw[looseness=0.9] (0,0) 
to[out=-90,in=-180] (1,-0.5)
to[out=0,in=-180] (2,-0.4)
to[out=0,in=-180] (3,-0.5)
to[out=0,in=-90] (4,0)
to[out=90,in=0] (3,0.5)
to[out=-180,in=0] (2,0.4)
to[out=-180,in=0] (1,0.5)
to[out=-180,in=90] (0,0);
\jtqfthandle{1}{0};
\jtqfthandle{3}{0};
\node at (2.2*\w,0) [font=\footnotesize] {$\dots$};
} .
$$

Then the infinite sets $\{\phi^i\}_{i\ge0}$ in $\cC([1],[1])$ and $\{\hat\phi_i\}_{i\ge0}$ in $\cC([0],[0])$ are each linearly independent.

Given scalars $\alpha=(\alpha_i)_{i\ge0}$ in $\kk$ and a polynomial $u=u(x)=\sum_{i\ge0} u_i x^i \in\kk[x]$, let $\cI$ be the tensor ideal in $\cC$ generated by the morphisms
\begin{equation} \label{eq:rels-Cob}
(\hat\phi_i - \alpha_i \id_{[0]})_{i\ge 0} \,,\qquad u(\phi) \,.     
\end{equation}
Then $\Cob_{\alpha,u}:=(\cC/\cI)^\text{ps}$ is a hom-finite symmetric pseudo-tensor category in the sense of \Cref{bg:pseudo}, where the braiding is defined using $\Psi$. The composition in $\Cob_{\alpha,u}$ is still defined by gluing along boundary circles, but now any closed component of genus $i$ is replaced by a scalar factor $\alpha_i$.

The two relations in \Cref{eq:rels-Cob} together imply the relation
$$
\Big(\sum_{i\ge0} u_i \alpha_{i+j} \Big) \id_{[0]} = 0 \qquad \tforall j\ge0 .
$$
Hence, the hom-finite category $\Cob_{\alpha,u}$ can have non-zero morphisms only if the sequence $(\alpha_i)_i$ is recurring, or equivalently, the power series with coefficients $\alpha$ is a rational function, and $u$ is divisible by a unique minimal monic polynomial $u_\alpha$ which encodes the shortest recurrence relation for the sequence $\alpha$.

\begin{definition} \label{def:Frob-alg-datum}
A \emph{Frobenius algebra datum} is a pair $(\alpha,u)$ consisting of scalars $\alpha=(\alpha_i)_{i\ge0}$ in $\kk$ and a monic polynomial $u=u(x)=\sum_{i\ge0} u_i x^i$ in $\kk[x]$ such that $\sum_{i\ge0} u_i \alpha_{i+j}=0$ for all $j\ge0$.
\end{definition}

It turns out that the object $[1]$ in $\Cob_{\alpha,u}$ is a commutative Frobenius algebra object (which we call the \emph{canonical Frobenius algebra}) whose structural morphisms $(\eta,\mu,\varepsilon,\Delta)$, as above, satisfy the additional identities
\begin{equation}
\label{eq:identities-Frob}
\varepsilon(\mu\Delta)^i\eta = \alpha_i
\quad\tforall i\ge0 ,\qquad
u(\mu\Delta) = 0 .
\end{equation}
We say a Frobenius algebra $A$ \emph{has Frobenius algebra} datum $(\alpha,u)$ whenever its structural morphisms $(\eta,\mu,\varepsilon,\Delta)$ satisfy \Cref{eq:identities-Frob}. Then $\Cob_{\alpha,u}$ is the universal symmetric pseudo-tensor category with a braiding defined by $\Psi$ and with a commutative Frobenius algebra with Frobenius algebra datum $(\alpha,u)$. That is, the category of $\kk$-linear symmetric monoidal functors from $\Cob_{\alpha,u}$ to any symmetric pseudo-tensor category $\cD$ is equivalent to the category of commutative Frobenius algebra objects in $\cD$ whose structural morphisms satisfy \Cref{eq:identities-Frob}. This follows from \cite{KOK}*{Section~2.3} or \cite{FLP}*{Proposition~3.7}.

\subsection{Categories defined using partition diagrams} \label{bg:RepSt}

Fix a scalar $t\in\kk$. An important special case of the categories discussed in \Cref{bg:Cob} are the categories $\cS_t := \Cob_{\alpha,u}$ for $\alpha=(t,t,\dots)$ and $u(x)=x-1$. It is easy to see that $(\alpha,u)$ defined in this way is a Frobenius algebra datum in the sense of \Cref{def:Frob-alg-datum}.

Let $f$ be any cobordism and let $f'$ be the same cobordism but with all handles removed. Then as $u(x)=x-1$, the relations \Cref{eq:rels-Cob} imply that $f=f'$ in $\cS_t$. Similarly, the relations in $\Cob_{\alpha,u}$ allow us to remove closed components from $f$ and replace them with scalars according to the parameters $(\alpha_i)_i$. Hence, we lose nothing thinking of the cobordisms in $\cS_t$ as string diagrams with a given number of upper and lower points instead of circles, which are connected by a number of strings, where string diagrams are considered equivalent if the partition of the set of upper and lower points defined by connectedness via strings coincide. Composition in $\cS_t$ is given by gluing along the boundary points and replacing any closed components by a factor $t$. We denote the set of string diagrams, or \emph{partition diagrams}, with $m$ upper and $m'$ lower points by $P_{m,m'}$.

As a specialization of the universal property of $\Cob_{\alpha,u}$ (see \Cref{bg:Cob}), $\cS_t$ is the universal symmetric pseudo-tensor category with a commutative Frobenius algebra object of dimension $t$ which is \emph{special}, that is, its structural morphisms $(\eta,\mu,\varepsilon,\Delta)$ satisfy $\mu\Delta=\id$.

Whenever $t=n\in\NN_0$, the universal property immediately implies the existence of a $\kk$-linear symmetric monoidal functor 
$$
\cS_n\to\Rep S_n
$$
sending $[1]$ to the standard representation $\kk^n$, which is a suitable Frobenius algebra in $\Rep S_n$. If $\kk$ is of characteristic $0$, this functor is full and essentially surjective (and, moreover, the \emph{semisimplification} functor as studied in \cite{EO-semisimplification}), which is why the categories $\cS_t$ can be considered interpolation categories for the symmetric groups.

For any partition diagram, we assign as a \emph{size} to each of its components the number of points contained in it. We call a component of a partition diagram \emph{even} or \emph{odd} depending on (the parity of) its size. We also will also call a component \emph{upper} or \emph{lower} if it contains only upper or lower points, respectively.

Beyond $\cS_t$, we consider certain subcategories of it. Given a suitable subset $Q\subset\bigsqcup_{m,m'} P_{m,m'}$ of partition diagrams, we can associate to, first, the $\kk$-linear monoidal subcategory of $\cS_t$ on all objects of the form $[m]$, for $m\ge0$, with hom-spaces formed by the linear combinations of partition diagrams in $Q$. Here, ``suitable'' means that the subset $Q$ has to have closure properties with respect to the composition and the tensor product in $\cS_t$. We can then associate to $Q$ the pseudo-abelian completion of this $\kk$-linear monoidal subcategory. In this way, a large number of pseudo-tensor subcategories of $\cS_t$ arise, including the following two:
\begin{itemize}
\item $\cH_t$: for $Q$ the set of partition diagrams all of whose components are even,
\item $\cS'_t$: for $Q$ the set of partition diagrams with an even number of odd components.
\end{itemize}

These categories can also be regarded as interpolation categories. For any $n\ge0$, let $H_n=S_2\wr S_n$ be the \emph{hyperoctahedral group} and let $S'_n:=S_2\times S_n$ be the group sometimes called \emph{modified symmetric group}. Then we have linear symmetric monoidal functors
$$
\cH_n\to\Rep H_n,\qquad \cS'_n\to\Rep S'_n, \qquad \tforall n\ge0,
$$
sending $[1]$ to the $n$-dimensional reflection representation of $H_n$, and to the $n$-dimensional reflection representation of $S'_n=S_2\times S_n$ on which $S_n$ acts by permuting a basis and $S_2$ acts by a sign, respectively. In characteristic $0$, these functors are full and essentially surjective. The categories $\cH_n$ and $\cS'_n$ for $n\ge0$ are discussed in \cite{BanicaSpeicher-liberation}, the categories $\cH_t$ and $\cS'_t$ for general $t$ fall into the framework developed in \cite{FM}. 

The categories $\cS_t$ were introduced by Deligne in \cite{Del} and are often denoted $\RepSt$. Different aspects of these categories were studied in \cites{CO-blocks,CO-ab,FHL}. The categories $\cH_t$ have also been called $\RepHt$ and have been studied in \cites{FM, Heidersdorf-Tyriard}. Various subcategories of $\cS_t$ corresponding to suitable subclasses of partition diagrams, including the mentioned ones, were systematically studied in \cite{FM}, building in part on the theory of so-called ``easy orthogonal quantum groups'', e.g.~in \cite{BanicaSpeicher-liberation}. The subcategory of $\cS_t$ generated by non-crossing partition diagrams of component size $2$ are, by definition, pseudo-abelian completions of \emph{Temperley--Lieb categories}.

\section{Pseudo-diagrammatic categories}
\label{sec:pdiag-categories}

We fix a field $\kk$. $\cC$ will always denote at least a $\kk$-linear monoidal category in this section. 

\subsection{Exactness conditions} \label{sec:exactness}

We start by defining a class of $\kk$-linear monoidal categories whose hom-spaces have bases that behave well under tensor products.

\begin{definition} \label{def:pseudo-diagrammatic} A \emph{pseudo-diagrammatic category} is a $\kk$-linear monoidal category $\cC$ with the following additional condition: 
\begin{itemize}
\item[(Diag)] \label{cond:diag} 
There are bases $B_{U,V}$ of the hom-spaces $\cC(U,V)$, for all $U,V\in \cC$, such that

(a) the tensor product in $\cC$ restricts to an injective map 
$$
B_{U,V}\times B_{U',V'}\hookrightarrow B_{U\o U',V\o V'} \qquad\tforall U,U',V,V'\in\cC,
$$

(b) whenever $b\o b'$ factors through $\one$ for basis elements $b,b'$, then both $b$ and $b'$ factor through $\one$.
\end{itemize}
\end{definition}

Note that a factorization of a morphism via the object $\one$ can also be interpreted as a tensor product decomposition of this morphism. Hence, the conditions in \hyperref[cond:diag]{(Diag)} can be verified without using the composition in $\cC$.

\begin{definition} \label{def:ex} We say a $\kk$-linear monoidal category $\cC$ has property...
\begin{itemize}
\item[(Ex1)] \label{cond:ex1} if the tensor product in $\cC$ produces an injective $\kk$-linear map 
$$ \psi_{U,V}\:\cC(\one,V)\o_\kk\cC(U,\one)\hookrightarrow\cC(U,V) \quad\tforall U,V\in\cC;
$$

\item[(Ex2)] \label{cond:ex2} whenever a tensor product $f\o g$ of morphisms in $\cC$ factors through $\one$, then both $f$ and $g$ are a sum of morphisms factoring through $\one$.
\end{itemize}  
\end{definition}

Note that the maps $\psi_{U,V}$ in \hyperref[cond:ex1]{(Ex1)} could equivalently be defined using the composition in $\cC$. 

\begin{remark} A sum of morphisms factoring through $\one$ does not have to factor through $\one$: for instance, consider ${}^{\bullet\bullet}_{\bullet\bullet}+{}^{\cup}_{\cap}$ in $\cS_t$ or $\mymatrix{ 1&0\\0&1 }=\mymatrix{ 1&0\\0&0 }+\mymatrix{ 0&0\\0&1 }$ in the (tensor) category of vector spaces.
\end{remark}

\begin{lemma} \label{lem:diag-v-ex12}
    Any pseudo-diagrammatic category has properties \hyperref[cond:ex1]{(Ex1)} and \hyperref[cond:ex2]{(Ex2)}.
\end{lemma}

\begin{proof}
\hyperref[cond:diag]{(Diag)}(a) immediately implies \hyperref[cond:ex1]{(Ex1)}.

To show \hyperref[cond:ex2]{(Ex2)}, consider a tensor product $h=f\o g$ which factors through $\one$, say $h=v u = v\o u$, where the source of $v$ and the target of $u$ is $\one$. We can write $f,g,u,v$ in terms of the bases whose existence is the subject of \hyperref[cond:diag]{(Diag)}, then
$$
h = \sum_{i,j} r_{ij} b_i\o b'_j = \sum_{k,\ell} s_{k\ell} b''_k \o b'''_\ell  
$$
for some scalars $r_{ij},s_{k\ell}\in\kk$, where the first sum corresponds to $f\o g$ and the second sum corresponds to $v\o u$. In particular, the source of $b''_k$ and the target of $b'''_\ell$ is $\one$, for all $k,\ell$. Due to \hyperref[cond:diag]{(Diag)}(a), the tensor products $b_i\o b'_j$ and $b''_k\o b'''_\ell$ are elements of the same basis, so if $r_{ij}\neq0$, then $b_i\o b'_j$ is of the form $b''_k\o b'''_\ell$ for some $k,l$. In particular, $b_i\o b'_j$ factors through $\one$. But then, due to \hyperref[cond:diag]{(Diag)}(b), both $b_i$ and $b'_j$ factor through $\one$, as desired.
\end{proof}

We want to extend the notion ``pseudo-diagrammatic'' to additive or pseudo-abelian $\kk$-linear monoidal categories.

\begin{lemma} \label{lem:additive-pseudo-diagrammatic} Assume $\cC$ is a $\kk$-linear monoidal category, and $\cC_0\subset \cC$ is a full $\kk$-linear monoidal subcategory such that any object in $\cC$ is a direct summand in a finite direct sum of objects from $\cC_0$. Then \hyperref[cond:ex1]{(Ex1)} and \hyperref[cond:ex2]{(Ex2)} hold for $\cC$ if and only if they hold for $\cC_0$.
\end{lemma}

\begin{proof} Using fullness, it follows immediately that each of the two conditions holds for $\cC_0$ once it holds for $\cC$.

Now assume both hold for $\cC_0$. Then \hyperref[cond:ex1]{(Ex1)} for $\cC$ follows from observing that
$$
\psi_{\oplus_i U_i,\oplus_j V_j} = \oplus_{i,j} \psi_{U_i,V_j}
$$
for any finite sets of objects $(U_i)_i$ and $(V_j)_j$ in $\cC$. \hyperref[cond:ex2]{(Ex2)} for $\cC$ follows from noting that if any morphism between finite direct sums in $\cC$ given by entries $(f_{ij})_{i,j}$ factors through $\one$, then all $f_{ij}$ factor through $\one$; and that a direct sum of a morphism $f$ and a zero morphism factors through $\one$ if and only if $f$ factors through $\one$.
\end{proof}

\begin{definition}\label{def:pseudo-diagrammatic-plus}
An \emph{$\oplus$-pseudo-diagrammatic} category is a $\kk$-linear monoidal category $\cC$ which has a full $\kk$-linear monoidal subcategory $\cC_0$ that is pseudo-diagrammatic (in the sense of \Cref{def:pseudo-diagrammatic}) such that any object in $\cC$ is a direct summand in a finite direct sum of objects from $\cC_0$.
\end{definition}

\begin{corollary} \label{cor:add-pseudo-diagrammatic}
    Any $\oplus$-pseudo-diagrammatic has properties \hyperref[cond:ex1]{(Ex1)} and \hyperref[cond:ex2]{(Ex2)}.
\end{corollary}

\begin{proof} By \Cref{lem:diag-v-ex12} with \Cref{lem:additive-pseudo-diagrammatic}.
\end{proof}

\medskip

\begin{definition} \label{def:U-Uex} For any $\kk$-linear monoidal $\cC$, following  \cite{CEOP}*{Section~2.3}, we define $\cU(\cC)$ as the collection of non-zero morphisms $u$ of the form $u\in\cC(U,\one)$ for some object $U\in\cC$, and $\Uex(\cC)$ as the subset of those $u\in\cU(\cC)$ which are a coequalizer of $u\o U$ and $U\o u$, i.e., a cokernel of $u\o U-U\o u$.    
\end{definition}

If $\cC$ is a pseudo-tensor category, then the condition $\U(\cC)=\Uex(\cC)$ is necessary for the existence of an abelian envelope for $\cC$: in fact, the following is true (sufficient conditions will be discussed in \Cref{sec:ab-env-diag}):

\begin{proposition}[\cite{CEOP}*{Proposition~2.3.3}] \label{prop:U-Uex} If a pseudo-tensor category $\cD$ admits a fully faithful $\kk$-linear monoidal functor to a tensor category, then $\U(\cD)=\Uex(\cD)$.
\end{proposition}

Assume $\cC$ is $\kk$-linear rigid monoidal. For objects $U,V\in\cC$, let $\coev_U\:\one\to U\o U^*$ be the coevaluation morphism and define the linear map
$$
\theta_{U,V}\:\cC(U,V)\to\cC(\one,V\o U^*), \quad g\mapsto (g\o U^*) \coev_U 
\qquad\tforall U,V\in\cC.
$$

\begin{proposition} \label{prop:U-equals-Uex} Assume a $\kk$-linear rigid monoidal category $\cC$ has properties  \hyperref[cond:ex1]{(Ex1)} and \hyperref[cond:ex2]{(Ex2)}. Then $\U(\cC)=\Uex(\cC)$. In particular, in any pseudo-diagrammatic or $\oplus$-pseudo-diagrammatic category $\cC$ that is rigid, we have $\U(\cC)=\Uex(\cC)$.
\end{proposition}

\begin{proof} Being pseudo-diagrammatic or $\oplus$-pseudo-diagrammatic each implies \hyperref[cond:ex1]{(Ex1)} and \hyperref[cond:ex2]{(Ex2)} by \Cref{lem:diag-v-ex12} or \Cref{cor:add-pseudo-diagrammatic}, respectively, which reduces the second assertion to the first one.

So assume $\cC$ has properties \hyperref[cond:ex1]{(Ex1)} and \hyperref[cond:ex2]{(Ex2)}. Let $U\xrightarrow{u}\one$ and $U\xrightarrow{f} V$ be morphisms in $\cC$ such that 
$$ f(u\otimes U)=f(U\otimes u), 
\quad\text{i.e.,}\quad
u\o f = f\o u . $$
Then applying the linear mapping $g\mapsto (g\o U^*)(U\o\coev_U)$ on both sides yields
$$ \theta_{U,V}(f) u
 = f\o \theta_{U,\one}(u) .
$$ 
The factorization on the left-hand side shows that this morphism factors through $\one$. Now \hyperref[cond:ex2]{(Ex2)} implies that $f$ can be written as $\sum_{i=1}^t v_i u_i$ for some $t\ge0$ with morphism $u_i\in\cC(U,\one)$ and $v_i\in\cC(\one,V)$. Without loss of generality, we may pick the $(u_i)_i$ linearly independent and such that $u_1=u$. Then
$$
\theta_{U,V}(f)u = \sum_{i=1}^t (v_i\o\theta_{U,\one}(u)) u_i ,
$$
and \hyperref[cond:ex1]{(Ex1)} implies 
$$
\theta_{U,V}(f) \o_\kk u = \sum_{i=1}^t (v_i\o\theta_{U,\one}(u)) \o_\kk u_i.
$$
Comparing coefficients of $u_i$, for $i>1$, we get
$$ 0 = v_i\o \theta_{U,\one}(u) ,
$$
to which we apply the linear mapping $g\mapsto (V\o ev_U)(g\o U)$, yielding
$$
0 = v_i u
$$
so $v_i=0$, again using \hyperref[cond:ex1]{(Ex1)}. This means $f=v_1 u$, and $v_1$ is uniquely determined, once more by \hyperref[cond:ex1]{(Ex1)}. This means $u$ is a coequalizer of $U\o u$ and $u\o U$.
\end{proof}

\subsection{Abelian envelopes for pseudo-diagrammatic categories} \label{sec:ab-env-diag}

Having seen that pseudo-diagrammatic categories satisfy a certain necessary condition for the existence of abelian envelopes in \Cref{prop:U-equals-Uex}, we come to sufficient conditions now.

\begin{definition} \label{def:splitting} A morphism $f$ of a category $\cC$ is called \emph{split} if there is a morphism $g$ in $\cC$ such that $f=fgf$. Assume $\cC\subset\cD$ are additive monoidal categories. A \emph{left splitting object in $\cD$} of a morphism $f$ in $\cC$ is an object $0\not\cong X\in\cD$ such that $X\o f$ is split. Right splitting objects are defined similarly. We say $\cC$ \emph{has splitting objects in $\cD$} if for any morphism in $\cC$ there is a left and a right splitting object in $\cD$.
\end{definition}

We will say an additive monoidal category $\cC$ \emph{has splitting objects} if it has splitting objects in $\cC$. Another situation that will be relevant for us will come from the embedding $\cC\subset\Ind(\cC)$ of a category $\cC$ in its ind-completion.

\begin{example}\label{ex:semisimple-splitting} (a) A monomorphism $f$ is split if and only if there is a morphism $g$ such that $gf$ is an identity. An epimorphism $f$ is split if and only if there is a morphism $g$ such that $fg$ is an identity. In an abelian category, a morphism is split if and only if both factors in an epi-mono-factorization are split. Hence, in a semisimple category, any morphism is split, and in a semisimple tensor category, the tensor unit, or any object, is a splitting object for any morphism. 

(b) More generally, any projective object in a tensor category is a splitting object for any morphism:  Indeed, let $f$ be a morphism in a tensor category with image factorization $f=f_2f_1$ and cokernel $k$. Then the epimorphisms $f_1$ and $k$ become epimorphisms between projective objects, hence split, upon tensoring with any projective object. As $k$ forms a short exact sequence with $f_2$, this implies that $f_2$ becomes split upon tensoring with any projective object. Hence, $f$ becomes split upon tensoring with any projective object.

(c)  In fact, assume $P$ is a splitting object for all morphisms in a tensor category. Then 
$$ \Hom(-,P^*) \cong \Hom(-\o P, \one)
$$
is an exact functor, as any functor preserves split exact sequences, so $P^*$ is injective, hence, $P$ is projective. With (b), it follows that an object in a tensor category is a splitting object iff it is projective.
\end{example}

Splitting objects play an important role for the existence of abelian envelopes. An abelian envelope is said to have the \emph{quotient property} if any object in the abelian envelope is a quotient of an object in the original category.

\begin{theorem}[\cite{CEOP}*{Theorem~3.2.1, Remark~3.2.2}] \label{abelian-envelopes}  A pseudo-tensor category $\cD$ has an abelian envelope with the quotient property if $\U(\cD)=\Uex(\cD)$ and 

(a) it has splitting objects, or 

(b) it has a braiding and splitting objects in its ind-completion.
\end{theorem}

\begin{remark} \cite{CEOP}*{Theorem~3.2.1} has even weaker assumptions on $\cD$, which we do not use here. In particular, in situation (b), instead of $\cD$ having a braiding, it suffices that the class of morphisms $\cU(\cD)$ is what is called \emph{permutable} in \cite{CEOP}. This follows, for instance, if the forgetful functor from the monoidal center $\cZ(\cD)$ to $\cD$ is essentially surjective (see \cite{CEOP}*{Example~3.1.2}). The latter follows from $\cD$ having a braiding.
\end{remark}

\begin{example} \label{expl:projectives} Let $\cC$ be a tensor category with enough projectives, let $\cD$ be the pseudo-abelian subcategory generated by all projective objects and the tensor unit in $\cC$. Then $\cD$ satisfies the condition $\U(\cD)=\Uex(\cD)$ by \Cref{prop:U-Uex} and has splitting objects by \Cref{ex:semisimple-splitting}(b). Hence it has an abelian envelope by \Cref{abelian-envelopes}. In fact, this envelope is $\cC$, as follows from \cite{CEOP}*{Theorem~2.2.1}. 
    
\end{example}

\begin{corollary} \label{cor:abenv-for-diag}
A pseudo-tensor category $\cD$ with properties \hyperref[cond:ex1]{(Ex1)} and \hyperref[cond:ex2]{(Ex2)} as in \Cref{def:ex} has an abelian envelope with the quotient property if 
\begin{enumerate}
    \item [(a)] it has splitting objects, or 
    \item [(b)] it has a braiding and splitting objects in its ind-completion. 
\end{enumerate}
In particular, this applies to any $\oplus$-pseudo-diagrammatic pseudo-tensor category.
\end{corollary}

\begin{proof}
    The necessary exactness condition $\U(\cD)=\Uex(\cD)$ is satisfied by \Cref{prop:U-equals-Uex}.
\end{proof}

\section{Families of pseudo-diagrammatic categories} \label{sec:families}

We will show next that several interesting families of categories are pseudo-diagrammatic.

\subsection{Categories based on cobordisms or partitions}

Recall from \Cref{bg:Cob}, that the hom-spaces $\Cob_{\alpha,u}([m],[n])$ in the cobordism categories $\Cob_{\alpha,u}$, for any fixed choice of the parameters $\alpha,u$, have bases given by sets of cobordisms.

\begin{proposition} \label{prop:diag-for-cob}
Let $\cC$ be any pseudo-tensor subcategory of $\Cob_{\alpha,u}$ containing the object $[1]$ such that the hom-spaces $\cC([m],[n])$ are spanned by $\Cb_{m,n}\cap\cC([m],[n])$. Then $\cC$ is $\oplus$-pseudo-diagrammatic.
\end{proposition}

\begin{proof}
From the definition of the composition in $\Cob_{\alpha,u}$ it follows that, for all $m,n\ge0$, the subsets $\Cb'_{m,n}\subset\Cb_{m,n}$ of cobordisms with no closed components and with all components of genus at most $\deg(u)-1$ form bases in $\Cob_{\alpha,u}([m],[n])$.
We set $B_{m,n}:=\Cb'_{m,n}\cap\cC([m],[n])$ for all $m,n$. By our assumptions, $B_{m,n}$ is a basis of $\cC([m],[n])$. 

Recall also that for $m,n,k,\ell\ge0$, the map $B_{m,n}\times B_{k,\ell}\to B_{m+k,n+\ell}$ induced by the tensor product is given by vertical concatenation of cobordisms. This is an injective map, it can be inverted on its image by picking from a cobordism those components whose boundary is formed by the leftmost or rightmost, respectively, upper and lower boundary circles, recovering its tensor product decomposition. 

If such a tensor product factors through $\one=[0]$, this means no connected component of the cobordism contains both upper and lower boundary circles. 

We have verified condition \hyperref[cond:diag]{(Diag)}.
\end{proof}

This directly applies to the subcategories of $\Cob_{\alpha,u}$ discussed in \Cref{bg:RepSt}.

\begin{corollary} \label{prop:U-Uex-St} The following pseudo-tensor categories are $\oplus$-pseudo-diagrammatic, and hence $\U=\Uex$ for each of them: $\RepSt$, $\RepOt$, pseudo-abelian completions of the Temperley--Lieb categories, $\RepHt$, all ``interpolating partition categories'' from \cite{FM}.
\end{corollary}

\subsection{Knop's tensor envelopes}

We recall some definitions from \cite{Knop}. Let $\cA$ be any regular category, in particular, $\cA$ has finite products and a form of image factorizations, as explained in the following.

Recall that a \emph{subobject} of an object $X\in\cA$ is an isomorphism class of monomorphisms into $X$. We write $X'\subset X$ to indicate there is a monomorphism $X'\to X$ in $\cA$, which is fixed but suppressed, representing a subobject of $X$ which we will sometimes call $X'$ by abuse of notation.
A \emph{relation} in $\cA$ between two objects $X,Y$ is a subobject of $X\x Y$. Let $R_{X,Y}$ denote the collection of relations of between $X$ and $Y$.

As $\cA$ is assumed regular, any morphism $f:X\to Y$ in $\cA$ factors as $f=X\xrightarrow{e} I\xrightarrow{i} Y$, where $e$ is an extremal epimorphism and $i$ is a monomorphism in $\cA$, and the factorization is unique up to isomorphisms. This factorization is called \emph{image factorization}, and the unique subobject represented by the monomorphism $i:I\to Y$ is called the \emph{image} of $f$, $\im(f)$.

We will require the following general property of regular categories.

\begin{lemma} \label{lem:product-of-subobjects} Let $X_1$, $X_2$, $Y_1$, $Y_2$ be objects and let $A\subset X_1\times X_2$, $B\subset Y_1\times Y_2$, $A'\subset X_1\times Y_1$, $B'\subset X_2\times Y_2$ be subobjects in  $\cA$. Assume $A\times B=A'\times B'$ as subobjects of $X_1\times X_2\times Y_1\times Y_2$ (after permutation of factors). Then there are subobjects $A_1\subset X_1$, $A_2\subset X_2$, $B_1\subset Y_1$, $B_2\subset Y_2$ such that $A'=A_1\times B_1$ and $B'=A_2\times B_2$ as subobjects.
\end{lemma}

\begin{proof}  We claim that in a regular category, the classes of monomorphisms and extremal epimorphisms are each closed under forming products of morphisms. Indeed, this is true for monomorphisms in any category. For the statement on extremal epimorphisms, it suffices to show that for each extremal epimorphism $\pi\:X\to Y$ and any object $Z$, the product $\id_Z\times\pi$ is an extremal epimorphism, as the class of extremal morphisms is closed under composition. Consider the diagram
$$
\begin{tikzcd}[column sep=2cm]
Z\times X \ar[r,"\id_Z\times \pi"] \ar[d,twoheadrightarrow] &
Z\times Y \ar[r,twoheadrightarrow] \ar[d,twoheadrightarrow] &
Z \ar[d,twoheadrightarrow] \\
X \ar[r,twoheadrightarrow,"\pi"] 
& Y \ar[r,twoheadrightarrow] &
*
\end{tikzcd} \qquad,
$$
where $*$ is the terminal object and all unlabeled two-headed arrows are terminal morphisms or structural morphisms of products. Both inner squares in the diagram commute. The right inner and the outer square can be identified as pullback squares. Hence by the pasting law, the left inner square is a pullback. This shows $\id_Z\times\pi$ is an extremal epimorphism, as those are closed under pullbacks (\cite{Knop}*{2.1~Definition, R3}). We have established the claim.

We return to the main assertion. Taking image factorizations, we have the following commutative diagrams consisting of monomorphisms and extremal epimorphisms, as indicated by the types of arrows:
$$
\begin{tikzcd}
A \ar[r,hookrightarrow,"\iota_A"] \ar[d,twoheadrightarrow,"\pi'_A"]
& X_1\times X_2 \ar[d,twoheadrightarrow,"\pi_A"]
\\
A_1 \ar[r,hookrightarrow,"\iota'_A"]
& X_1
\end{tikzcd}
\quad,\qquad
\begin{tikzcd}
B \ar[r,hookrightarrow,"\iota_B"] \ar[d,twoheadrightarrow,"\pi'_B"]
& Y_1\times Y_2 \ar[d,twoheadrightarrow,"\pi_B"]
\\
B_1 \ar[r,hookrightarrow,"\iota'_B"]
& Y_1
\end{tikzcd} 
\qquad,
$$
where we call the image objects $A_1$ and $B_1$, respectively.
This means we have a commutative diagram as follows:
$$
\begin{tikzcd}[column sep=2cm]
A\times B
    \ar[r,hookrightarrow,"\id\times\iota_B"]
    \ar[d,twoheadrightarrow,"\pi'_A\times\pi'_B"]
& A\times Y_1\times Y_2
    \ar[r,hookrightarrow,"\iota_A\times\id"]
    \ar[d,twoheadrightarrow,"\pi'_A\times\pi_B"]
& X_1\times X_2\times Y_1\times Y_2
    \ar[d,twoheadrightarrow,"\pi_A\times\pi_B"]
\\
A_1\times B_1
    \ar[r,hookrightarrow,"\id\times\iota'_B"]
& A_1\times Y_1
    \ar[r,hookrightarrow,"\iota'_A\times\id"]
& X_1\times Y_1
\end{tikzcd} ,
$$
where the types of morphisms can be determined using the above claim.

The diagram exhibits $A_1\times B_1$ as the image of $A\times B$ under $\pi_A\times\pi_B$, but this image is $A'$. The proof of the decomposition of $B'$ is similar.  
\end{proof}

Let $\cS(\cA)$ denote the class of extremal epimorphisms in $\cA$. A \emph{degree function} for $\cA$ is a map $\delta:\cS(\cA)\to\kk$ satisfying certain conditions (\cite{Knop}*{3.1.~Definition}).

For any regular category $\cA$ with a degree function $\delta$, the category $\cT^0(\cA,\delta)$ is defined as a $\kk$-linear monoidal category as follows:
\begin{itemize}
\item Its \emph{objects} are the objects of $\cA$.
\item Its \emph{hom-spaces} $\Hom(X,Y)$ are the $\kk$-vector spaces spanned by the relations $R_{X,Y}$ between $X$ and $Y$, for all $X,Y\in\cA$,
\item Its \emph{composition} is given by the pullbacks of spans, using the degree function. More precisely, for $R_1\subset X\x Y$ and $R_2\subset Y\x Z$, we obtain a morphism $f:R_1\x_Y R_2\to X\x Z$. Let $f=ie$ be its image factorization. Then the composition of $R_2\circ R_1$ in $\cT^0(\cA)$ is defined as the relation given by $i$, i.e., $\im(f)$, multiplied by the scalar $\delta(e)$.
\item Its \emph{tensor unit} $\one$ is given by the terminal object $*$ of $\cA$.
\item Its \emph{tensor product} is given by the product of subobjects.
\end{itemize}

The category $\cT(A,\delta)$ is defined as the pseudo-abelian envelope of $\cT^0(\cA,\delta)$, we call it \emph{Knop's tensor envelope} of $\cA$. It turns out to be rigid and symmetric.

Assume from now on that $\cA$ is essentially small and its terminal object $*$ has no non-trivial subobjects. Then $\cT^0(\cA,\delta)$ is a pseudo-tensor category as defined in \Cref{bg:pseudo}. Note that we have inclusions of categories $\cA\subset \cT^0(\cA,\delta)\subset\cT(\cA,\delta)$, where $f:X\to Y$ in $\cA$ is interpreted as its graph, a subobject of the form $X\subset X\times Y$.

\begin{proposition} Let $\cA$ be regular category, let $\delta$ be a degree function for $\cA$. Let $\cC$ be any pseudo-tensor subcategory of $\cT(\cA,\delta)$ containing all objects of $\cA$ such that the hom-spaces $\cC(X,Y)$ are spanned by $\cC(X,Y)\cap R_{X,Y}$, for all $X,Y\in\cA$. Then $\cC$ is $\oplus$-pseudo-diagrammatic.
\end{proposition}

\begin{proof} We verify the conditions \hyperref[cond:diag]{(Diag)} for $\cC^0$, the full $\kk$-linear monoidal subcategory of $\cC$ on the objects of $\cA$.

Recall that the tensor product in $\cC^0$ is given by the product of subobjects.
Consider objects $X_1,X_2,Y_1,Y_2$ in $\cA$ and subobjects $R_i\subset X_i\times Y_i$, for $i=1,2$. Set $R:=R_1\x R_2$, a subobject of $X_1\x X_2\x Y_1\x Y_2$. Then $R_i$ as a subobject of $X_i\x Y_i$ is the image of the morphism $r:R\to X_1\x X_2\x Y_1\x Y_2$ post-composed with a suitable combination of projections. In particular, it is completely determined by the morphism $r$. This proves \hyperref[cond:diag]{(Diag)}(a).

Recall that the tensor unit in $\cC^0$ is given by the terminal object $*\in\cA$. 
Assume a subobject $R\subset X_1\x X_2\x Y_1\x Y_2$ corresponds to a tensor product of morphisms given by relations $R_i\subset X_i\x Y_i$ as above, and assume at the same time the morphism given by $R$ factors via the tensor unit. The latter means there are relations
$$
S_1\subset X_1\x X_2\x *,\quad
S_2\subset *\x Y_1\x Y_2 ,
$$
such that $R=S_1\times_{*}S_2$ is given by a pullback along the projections of $S_1,S_2$ to the terminal object, i.e., the middle square in 
$$\begin{tikzcd}[row sep=4pt]
&&R\ar[dr]\ar[dl]&&\\
&S_1\ar[dr]\ar[dl]&&S_2\ar[dr]\ar[dl]&\\
X_1\times X_2&& {*} &&Y_1\times Y_2
\end{tikzcd}$$
is a pullback. Note that since monomorphisms are closed under taking direct products, the resulting span above is given by an injective map $R\to X_1\times X_2\times Y_1\times Y_2$.
This just means that $R$ is the product of the subobjects $S_1,S_2$. 
But this means that each $R_i$, which is obtained from $R$ via suitable projections, is a product of suitable projections of the $S_i$, by \Cref{lem:product-of-subobjects}. This means the morphisms represented by $R_i$ factor via the tensor unit, proving \hyperref[cond:diag]{(Diag)}(b).
\end{proof}

In \cite{Knop}, Knop obtains a realization of $\cS_t$ as in \Cref{bg:RepSt} by specializing $\cA$ to the opposite of the category of finite sets. He also obtains interpolation categories $\RepGLnFq$ for the general linear groups over finite fields $\F_q$ from specializing $\cA$ to be the category of finite-dimensional $\F_q$-vector spaces (but note that $\RepGLnFq$ is a $\kk$-linear category).

\begin{corollary} \label{prop:U-Uex-GlnFq}
In particular, the pseudo-tensor category $\cD=\RepGLnFq$ is $\oplus$-pseudo-diagrammatic, and thus, $\U(\cD)=\Uex(\cD)$.
\end{corollary}

The categories $\RepGLnFq$ and their abelian envelopes are studied in \cite{EAH-glt} and in \cite{HS}*{Section~15}.

\subsection{Harman--Snowden's categories from pro-oligomorphic groups} We recall some definitions from \cite{HS}.
Let $G$ be a pro-oligomorphic group, a certain kind of topological group. A $G$-set is called \emph{finitary} if it has finitely many $G$-orbits. A $G$-set is called \emph{smooth} if the stabilizer of any point is an open subgroup of $G$. A \emph{measure} for $G$ is essentially a function $\mu$ assigning scalars in $\kk$ to finitary $G$-sets, subject to certain conditions (see \cite{HS}*{Section~3}).

Given a pro-oligomorphic group $G$ with a measure $\mu$, the $\kk$-linear monoidal category $\uPerm(G,\mu)$ is defined as follows:
\begin{itemize}
\item Its \emph{objects} are the smooth, finitary, possibly empty $G$-sets.
\item Its \emph{hom-spaces} $\Hom(X,Y)$ are the subspaces of $G$-invariant functions $X\times Y\to\kk$ which are spanned by the indicator functions on $G$-orbits of $X\x Y$.
\item Its \emph{composition} is given by a convolution product. Concretely, for $G$-orbits $A\subset X\x Y$ and $B\subset Y\x Z$, the composition of the corresponding indicator functions is defined by
$$
(1_B \circ 1_A)(x,z) = \mu(\{ y\in Y: (x,y)\in A, (y,z)\in B \})
\quad \tforall x\in X,z\in Z.
$$
\item Its \emph{tensor unit} is the singleton $G$-set $*$.
\item Its \emph{tensor product} is described on indicator functions for $G$-orbits $A\subset X_1\x Y_1$, $B\subset X_2\x Y_2$ as
$$
1_A\o 1_B = 1_{A\times B} .
$$
\end{itemize}

It turns out that the category $\uPerm(G,\mu)$ is additive with $X\oplus Y=X\sqcup Y$, the disjoint union of $G$-sets, and rigid.

Fix a pro-oligomorphic group $G$ with a measure $\mu$. Set $\cC:=\uPerm(G,\mu)$. Define $B_{X,Y}$ as the set of indicator functions on the $G$-orbits of $X\x Y$ in $\cC(X,Y)$, for all $X,Y$.

\begin{proposition} Let $\cD$ be a wide $\kk$-linear monoidal subcategory of $\cC=\uPerm(G,\mu)$ such that the hom-spaces $\cD(X,Y)$ are spanned by $\cD(X,Y)\cap B_{X,Y}$, for all $X,Y$. Then $\cD$ is $\oplus$-pseudo-diagrammatic.
\end{proposition}

\begin{proof}
First note that the tensor product of two indicator functions is an indicator function by definition. More precisely, if $A\subset X_1\x Y_1$ and $B\subset X_2\x Y_2$ are $G$-orbits, then $1_A \o 1_B = 1_{A\x B}$. As $A$ and $B$ can be recovered from $A\times B$, for any fixed $X_1,X_2,Y_1,Y_2$, this formula shows \hyperref[cond:diag]{(Diag)}(a).

Consider finitary $G$-sets $X_1,X_2,Y_1,Y_2$. Then the image of the composition map, or equivalently, the tensor product map
$$
\cC(X_1\x X_2,\one) \x \cC(\one,Y_1\x Y_2) \to \cC(X_1\x X_2,Y_1\x Y_2), \quad
(1_A,1_B)\mapsto 1_{A\x B}
$$
is by definition the span of the indicator functions $1_{A\times B}$ for $A\subset X_1\x X_2$ and $B\subset Y_1\x Y_2$. Assume that for $G$-orbits $A'\subset X_1\x Y_1$, $B'\subset X_2\x Y_2$, the tensor product $1_{A'}\o 1_{B'}=1_{A'\x B'}$ is of the form $1_{A\x B}$, for $G$-orbits $A\subset X_1\times X_2$ and $B\subset Y_1\times Y_2$. Then it follows from \Cref{lem:product-of-subobjects} applied to the category of finite sets (or from elementary arithmetic of sets) that $A'=A_1\x A_2$ and $B'=B_1\x B_2$ for subobjects $A_1\subset X_1$, $A_2\subset Y_1$, $B_1\subset X_2$, $B_2\subset Y_2$. This shows \hyperref[cond:diag]{(Diag)}(b). 
\end{proof}

\section{Monoidal adjunctions}
\label{sec:mon-adj}

In \Cref{sec:ab-env-diag}, we have seen that in many cases of interest (see \Cref{cor:abenv-for-diag}), the existence of splitting objects implies the existence of abelian envelopes.
In this section, we will show that splitting objects can be obtained from monoidal adjunctions, and we will explain how such monoidal adjunctions can be established.

\subsection{Existence of adjoints} Recall that a functor $F\:\cC\to\cD$ has a right adjoint if and only if, for any object $Y\in\cD$, the functor $\cD(F(-),Y)$ (sometimes also called \emph{presheaf}) is representable by an object in $\cC$. Our goal is to show that, if $F$ is a monoidal functor satisfying certain conditions, it suffices to consider the object $Y=\one\in\cD$.

We begin with an auxiliary result. Recall our conventions regarding rigid categories from \Cref{bg:pseudo}.
\begin{lemma} \label{lemma:representable_of_image_object}
Let $F\:\cC\to\cD$ be a monoidal functor between rigid monoidal categories. Assume $\cD( F(-),\one)$ is representable by an object $Z\in\cC$. Then we have isomorphisms
$$
\cD(F(X),F(Y)) \cong \cC(X,Z\o Y) \cong \cC(Z^*\o X,Y)
$$
and
$$
\cD(F(X),F(Y)) \cong \cC(X, Y\o Z) \cong \cC(X\o {}^*Z,Y)
$$
that are natural in $X,Y \in \cC$.
It also follows that $Z^* \cong {}^* Z$.
\end{lemma}

\begin{proof} Using that monoidal functors are compatible with taking duals, we compute
\begin{align*}
\cD(F(X),F(Y)) &\cong \cD(F(X), \one\o F(Y)) 
 \cong \cD(F(X)\o{}^*F(Y), \one) \\
& \cong \cD(F(X)\o F({}^*Y), \one) 
 \cong \cD(F(X\o {}^*Y), \one) \\
& \cong \cC(X\o{}^*Y, Z) 
 \cong \cC(X,Z\o Y) \cong \cC(Z^*\o X,Y) 
\end{align*}
and similarly
\begin{align*}
\cD(F(X),F(Y)) &\cong \cD(F(X), F(Y)\o\one) 
\cong \cD(F(Y^*\o X), \one) \\
& \cong \cC(Y^*\o X, Z) 
 \cong \cC(X, Y\o Z) \cong \cC(X\o {}^*Z,Y) .
\end{align*}
Specializing $X$ to $\one$, we obtain $\cC(Z^*,-)\cong\cC({}^*Z,-)$, which implies $Z^*\cong{}^*Z$ using the Yoneda embedding.
\end{proof}

\begin{definition}
We call a linear functor $F\: \cC \rightarrow \cD$ between $\kk$-linear categories \emph{dominant} if for all $Y \in \cD$ there exists an $X \in \cC$ such that $Y$ is a direct summand of $F(X)$.
\end{definition}

\begin{example} \label{expl:tensor-generator}
Let $F\: \cC \rightarrow \cD$ be a linear monoidal functor between $\kk$-linear additive monoidal categories. Recall from \Cref{bg:pseudo} that a \emph{tensor generator} of $\cD$ is an object $T \in \cD$ such that every indecomposable object $Y \in \cD$ occurs as a direct summand of $T^{\otimes n}$ for some $n \in \ZZ_{\geq 0}$.
If $F$ maps some object of $\cC$ to an object containing a tensor generator of $\cD$ as a direct summand, then $F$ is dominant.
\end{example}

\begin{example} \label{expl:composition-dominant} The composition of dominant functors is dominant.
\end{example}

\begin{corollary}
\label{corollary:adjoints_rigid_case}
Let $F\: \cC \rightarrow \cD$ be a linear monoidal functor between $\kk$-linear additive rigid monoidal categories. Suppose that 
\begin{enumerate}
    \item $\cD(F(-), \one )$ is representable by an object $Z\in\cC$,
    \item $F$ is dominant, and
    \item $\cC$ is idempotent complete.
\end{enumerate}
Then $F$ has both a left adjoint $L$ and a right adjoint $R$ such that $R(\one)\cong Z$ and $L(\one)\cong Z^*\cong {}^*Z$.
\end{corollary}
\begin{proof}
For the right adjoint, consider $Y \in \cD$. We want to show that $\cD(F(-), Y)$ is representable.
Since $F$ is dominant, there is an $X \in \cC$ such that $Y$ is a summand of $F(X)$.
Hence, the functor $\cD(F(-), Y )$ is a direct summand of the functor $\cD( F(-), F(X) )$, and it suffices to show that $\cD( F(-), F(X) )$ is representable, since any direct summand of a representable functor is itself representable (by the Yoneda lemma and $\cC$ being idempotent complete). But representability of $\cD( F(-), F(X) )$ follows from \Cref{lemma:representable_of_image_object}, which also implies $R(\one)\cong Z$.

Similarly, for the left adjoint we have to show that $\cD(F(X), F(-))$ is representable for any $X\in\cC$. It follows from \Cref{lemma:representable_of_image_object} that indeed $Z^*\o X$ is a representing object.
\end{proof}

Recall that for us, a $\kk$-linear category is semisimple if it is abelian and every object is a finite direct sum of simple objects. In the semisimple case, we get the following stronger result.

\begin{corollary} \label{cor:adjoint-semisimple}
    Let $F\: \cC \rightarrow \cD$ be a $\kk$-linear monoidal functor between semisimple rigid monoidal categories $\cC$ and $\cD$. The following are equivalent:
    \begin{enumerate}
        \item $F$ has both a left adjoint $L$ and a right adjoint $R$ such that $L \cong R$.
        \item The functor $\cD( F(-), \one )$ is representable.
    \end{enumerate}
\end{corollary}
\begin{proof} (1) immediately implies (2). So assume (2), we want to show (1). 

Suppose $\cD( F(-), \one )$ is representable by an object $Z \in \cC$.
Consider a simple object $Y$ in $\cD$. We distinguish two cases.

\emph{First case:} There is no $X \in \cC$ such that $Y$ is a summand of $F(X)$.
Then $\cD( F(-), Y ) \cong 0$ is representable by the zero object.

\emph{Second case:} There is an $X \in \cC$ such that $Y$ is a summand of $F(X)$. Then representability of $\cD( F(-), Y )$ follows as in the proof of \Cref{corollary:adjoints_rigid_case}.

This shows $\cD(F(-), Y)$ is representable for all $Y\in\cD$, and hence, $F$ has a right adjoint $R$.

In any semisimple $\kk$-linear category, consider a finite set of pairwise non-isomorphic simple objects $(X_i)_i$ and two corresponding sets of multiplicities $(m_i)_i,(n_i)_i$. Then
\begin{equation}
\label{eq:multiplicities}    
\Hom\Big(\bigoplus_i X_i^{\oplus m_i},\bigoplus_i X_i^{\oplus n_i}\Big)
\cong \bigoplus_i \End(X_i)^{m_in_i}
\cong \Hom\Big(\bigoplus_i X_i^{\oplus n_i},\bigoplus_i X_i^{\oplus m_i}\Big)
\end{equation}
as vector spaces.

Hence, for any pair of objects $X\in\cC$, $Y\in\cD$, there are isomorphisms
$$
\cD(X,F(Y))\cong\cD(F(Y),X)\cong\cC(Y,R(X))\cong\cC(R(X),Y)
$$
of vector spaces, where the middle isomorphism comes from the previously discussed adjunction, and the other two isomorphisms are as in \Cref{eq:multiplicities}. As $\cD$ is semisimple, this means we have an isomorphism of functors $\cD(-,F(Y))\cong\cC(R(-),Y)$, for each $Y\in\cC$. This implies that $R$ is also a left adjoint of $F$. 
\end{proof}

\subsection{Adjunctions and ind-completions} 

Sometimes, we will have to pass to ind-completions to obtain adjoint functors.

For this section, let $\cC,\cD$ be $\kk$-linear rigid monoidal categories with ind-completions $\ti\cC:=\Ind(\cC)$ and $\ti\cD:=\Ind(\cD)$. Let $F\:\cC\to\cD$ be a dominant linear monoidal functor, let $\ti F:=\Ind(F)$ be the induced functor between the ind-completions, see \Cref{appendix:ind}. 

\begin{lemma} \label{lem:adjoint-ind} Assume there is $X\in\ti\cC$ such that 
$$
\ti\cD(\ti F(-),\one)|_{\cC^{\op}}\cong\ti\cC(-,X)|_{\cC^{\op}}.
$$
Then $\ti F$ has a right adjoint sending  $\one\mapsto X$.
\end{lemma}

\begin{proof} 
Let $\cD_0\subset\cD$ be the full subcategory spanned by isomorphism classes of objects in the image of $F$. As $F$ is dominant, the ind-completions of $\cD$ and $\cD_0$ coincide, so we may replace $\cD$ by $\cD_0$ and assume that $F$ is essentially surjective.

To show that $\ti F$ has a right adjoint, it is now enough to show that $\ti\cD(\ti F(-),F(Y))|_{\cC^{\op}}$ is ind-representable, for all $Y\in\cD$, by \Cref{lem:adj-ind}. 

As $\cC$ and $\cD$ are rigid, and $F$ is monoidal,
$$
\ti\cD(\ti F(-),F(Y))\cong \ti\cD(\ti F(-),\one)\circ (Y^*\o-),
$$
and $Y^*\o-$ has a right adjoint. Hence, it suffices to show $\ti\cD(\ti F(-),\one)$ is ind-representable, as asserted.
\end{proof}

\begin{lemma} \label{lem:adj-filtered} Assume $\cC=\bigcup_{i\ge0} \cC_i$, for a chain of full subcategories $\cC_0\subset\cC_1\subset\dots$ of $\cC$, and $\cD(F(-),\one)|_{\cC_i}$ is representable by an object $X_i\in\cC_i$, for all $i$.
Then there is a diagram
$$
X_0 \xrightarrow{h_0} X_1 \xrightarrow{h_1} \dots
$$
in $\cC$ such that $\ti F$ has a right adjoint sending $\one\mapsto\ti X:=\colim_i(X_i)$.
\end{lemma}

\begin{proof} By the assumptions, there are $p_i\in\cD(F(X_i),\one)$ such that the natural transformations
$$
\phi_i := p_i F(-) \:\quad
\cC(-,X_i)|_{\cC_i} \to \cD(F(-),\one)|_{\cC_i}
$$
are isomorphisms. Let $h_i$ be the preimage of $p_i$ under $(\phi_{i+1})_{X_i}$, for all $i$, then 
$$
p_i = p_{i+1} F(h_i) .
$$

By \Cref{lem:adjoint-ind}, it suffices to show that $\ti\cD(\ti F(-),\one)|_{\cC^{\op}}\cong\ti\cC(-,\ti X)|_{\cC^{\op}}$. Let $p\in\ti\cD(\ti F(\ti X),\one)$ be the morphism given by the family $(p_i)_i$. It is well-defined due to the above identity.

For any $i\ge0$ and $Y\in\cC_i$, the map
$$
p\ti F(-) \:\quad
\ti\cC(Y, \ti X) \to \ti\cD(\ti F(Y),\one)
$$
is an isomorphism, because it is described by
$$
\colim_{j \geq i} \left(\cC(Y, X_j) \xrightarrow{\phi_j|_Y}\ti\cD(\ti F(Y),\one)\right)
$$
where the maps $\phi_j|_Y$ are isomorphisms. The isomorphism is natural in $Y$, as it is of the form described by the Yoneda lemma.
\end{proof}

\subsection{Adjunctions and splitting objects}

We will see that adjunctions involving a monoidal functor can be used to find splitting objects (as in \Cref{sec:ab-env-diag}).

The following is a slight generalization of \cite{FLP2}*{Corollary~4.11, Section~4.4}.

\begin{proposition} \label{prop:projection-formulas}
Let $\adj F \cC \cD F'$ be a pair of functors such that $\cC$ and $\cD$ are monoidal, $F$ is monoidal, and $\cC'\subset\cC$ is a monoidal subcategory. Then there is an isomorphism of functors
$$
F'(X)\o Y\cong F'(X\o F(Y))
$$
on $\cC\times\cC'$ if 
\begin{itemize}
    \item $\cC'$ is left rigid and $F'$ is a left adjoint of $F$, or 
    \item $\cC'$ is right rigid and $F'$ is a right adjoint of $F$.
\end{itemize}

Similarly, there is an isomorphism of functors 
$$
X\o F'(Y)\cong F'(F(X)\o Y)
$$
on $\cC'\times\cC$ if
\begin{itemize}
    \item $\cC'$ is right rigid and $F'$ is a left adjoint of $F$, or 
    \item $\cC'$ is left rigid and $F'$ is a right adjoint of $F$.
\end{itemize}
\end{proposition}

\begin{proof} We only discuss the case that $\cC'$ is left rigid and $F'$ is a left adjoint of $F$, the other cases being similar. We have isomorphisms of functors
\begin{align*}
\cC(F'(X)\o Y,Z) 
&\cong \cC(F'(X), Z\o Y^*)
\cong \cD(X, F(Z)\o F(Y)^*) \\
&\cong \cD(X\o F(Y), F(Z))
\cong \cC(F'(X\o F(Y)), Z) ,
\end{align*}
on $\cC\op\times(\cC')\op\times\cC$, using that the image of a (left) dualizable object under a monoidal functor is (left) dualizable. This shows the assertion.
\end{proof}

In particular, this implies that splitting objects are transferred along monoidal adjunctions between rigid categories. Recall from \Cref{def:splitting} that a morphism $f$ is split if there is a morphism $g$ such that $f=fgf$, and that a left (right) splitting object for a morphism $f$ is a non-zero object $X$ such that $X\o f$ ($f\o X$) is split.

\begin{corollary} \label{cor:transport-splitting}
Let  $\adj{F}{\cC}{\cD}{F'}$ be a pair of functors such that $\cC$, $\cD$, $F$ are monoidal, $F'$ is a left \slash{} right adjoint of $F$, $\cC'\subset\cC$ is a monoidal subcategory, and $f$ is a morphism in $\cC'$. Then $F'(X)$ is a left splitting object of $f$ in $\cC$, for every left splitting object $X$ of $F(f)$ in $\cD$, if $\cC'$ is left \slash{} right rigid.

Similarly, $F'(X)$ is a right splitting object of $f$, for every right splitting object $X$ of $F(f)$, if $\cC'$ is right \slash{} left rigid. 

 Moreover, $F'(X)$ is a non-zero object if there is a non-zero morphism of the form $X\to F(Z)$ \slash{} $F(Z)\to X$ for some $Z\in\cC$.
 \end{corollary}

\begin{proof}
We only discuss the case that $F'$ is a left adjoint of $F$ and $\cC'$ is left rigid, the other cases being similar. By \Cref{prop:projection-formulas}, we have a natural isomorphism of functors
$$ F'(X)\o - \cong F'(X\o F(-))
$$
on $\cC'$. Now $F'(X)$ is a splitting object, as every functor preserves split morphisms, and $F'(X)$ is non-zero under the described assumption, as $\cC(F'(X),Z)\cong\cD(X,F(Z))\neq0$.
\end{proof}

\subsection{Abelian envelopes from monoidal adjunctions}
\label{sec:abelian-envelopes}
We record some consequences of this transfer of splitting objects regarding the existence of abelian envelopes.

\begin{theorem} \label{cor:ab-env-from-functor-semisimple}
Assume that
\begin{enumerate}
    \item [(a)] $\cC$ is an $\oplus$-pseudo-diagrammatic (in the sense of \Cref{def:pseudo-diagrammatic-plus}) pseudo-tensor category, 

\item[(b)] $\cD$ is a tensor category with enough projectives, and

\item[(c)] there is a linear monoidal functor $F\:\cC\to\cD$. 
\end{enumerate}
Assume also that
\begin{enumerate}
    \item [(d)] $F$ has a left \slash{} right adjoint $F'$ or

\item[(d')] $\cC$ is braided and $\Ind(F)$ has a left \slash{} right adjoint $F'$.
\end{enumerate}
Then for any projective object $P\in\cD$ with a non-zero morphism to \slash{} from $\one\in\cD$, $F'(P)$ is, for all morphisms in $\cC$, a non-zero left and right splitting object in $\cC$ or $\Ind(\cC)$, respectively, and in both cases, $\cC$ has an abelian envelope with the quotient property. 
\end{theorem}

\begin{proof} First note that $P$ is a splitting object in $\cD$ by \Cref{ex:semisimple-splitting}(b). Hence, $F'(P)$ is a non-zero (left and right) splitting object as asserted by \Cref{cor:transport-splitting}. This implies that $\cC$ has an abelian envelope with the quotient property by \Cref{cor:abenv-for-diag}.
\end{proof}

In particular, $\cD$ could even be semisimple, in which case $F'(\one)$ is a splitting object.

\begin{remark}
    We will see in \Cref{cor:criterion-enough-proj} below that under the assumption (d), the abelian envelope from \Cref{cor:ab-env-from-functor-semisimple} even has enough projectives.
\end{remark}

We also record a version of \Cref{cor:ab-env-from-functor-semisimple}, where only the images of the members of a filtration of $\cC$ under varying functors are assumed to be semisimple. 

\begin{theorem} \label{cor:ab-env-from-functor-semisimple-filtered}
Assume 
\begin{enumerate}
    \item [(a)] $\cC$ is an $\oplus$-pseudo-diagrammatic (in the sense of \Cref{def:pseudo-diagrammatic-plus}) pseudo-tensor category,

\item[(b)] $\cC=\bigcup_{i\ge0} \cC_i$, for full subcategories $\cC_0\subset\cC_1\subset\dots$, and 

\item[(c)] there are pseudo-tensor categories $(\cD_i)_{i\ge0}$ and linear monoidal functors $F_i:\cC\to\cD_i$ such that the full subcategory generated by $F_i(\cC_i)$ is semisimple, for all $i$.
\end{enumerate}
Assume also that 
\begin{enumerate}
    \item [(d)] $F_i$ has a left or right adjoint $F'_i$, for all $i$, or

\item[(d')] $\cC$ is braided and $\Ind(F_i)$ has a left or right adjoint $F'_i$, for all $i$.
\end{enumerate}
Then all morphisms in $\cC$ have splitting objects in $\cC$ or $\Ind(\cC)$, respectively, and in both cases, $\cC$ has an abelian envelope with the quotient property.
\end{theorem}

\begin{proof} Again, by \Cref{cor:adjoint-semisimple}, $F_i$ is both a left and right adjoint. By assumption, any morphism $f$ in $\cC$ is contained in one of the full subcategories $\cC_i$. Hence, $F_i'(\one)$ is a non-zero splitting object for $f$ in $\Ind(\cC)$ by \Cref{cor:transport-splitting}, since any morphism is split in the semisimple category $\cD_i$. This implies that $\cC$ has an abelian envelope with the quotient property by \Cref{cor:abenv-for-diag}.
\end{proof}

\section{A construction of the abelian envelope with enough projectives}

We slightly generalize and simplify the construction of abelian envelopes with enough projectives in \cite{BEO} to give a concrete realization of abelian envelopes in some of our examples.
Our construction uses the theory of finitely presented functors, see \Cref{appendix:fp} for a brief introduction.

Let $\cC$ be a pseudo-tensor category over a field $\kk$.
In this section, we will make use of objects $S \in \cC$ such that $S \otimes f$ and $f \otimes S$ are split for all morphisms $f$ in $\cC$. In other words, $S$ is both a left and right splitting object for any morphism in $\cC$ in the sense of \Cref{def:splitting}, and we refer to such objects as \emph{global splitting objects}.
We denote by $\cS \subseteq \cC$ the full subcategory spanned by all global splitting objects in $\cC$.
It is easy to see that $\cS$ is pseudo-abelian, i.e., it has finite direct sums and is closed under taking finite direct summands.
Moreover, $\cS$ is closed under taking duals in $\cC$, and if $A \in \cC$ and $S \in \cS$, then both $A \otimes S$ and $S \otimes A$ are in $\cS$.

It is the goal of this section to prove the following theorem.
\begin{theorem} \label{thm:ab-env-projectives}
The following are equivalent: 
\begin{enumerate}
    \item[(a)] The category $\cC$ has a monoidal abelian envelope with enough projectives.
\item [(b)] The tensor unit $\one$ of $\cC$ has a presentation by global splitting objects, i.e., there is a cokernel sequence
$$
S \rightarrow T \rightarrow \one \rightarrow 0
$$
with $S, T \in \cS$.
\item [(c)] $\cC$ has a non-zero global splitting object and $\cU(\cC)=\Uex(\cC)$ (see \Cref{def:U-Uex}).
\end{enumerate}
In this case, the monoidal abelian envelope is given by the category of finitely presented functors $\rmod{\cS}$ together with the functor
$$(A \mapsto \cC(-,A)|_{\cS}): \cC \rightarrow \rmod{\cS}.$$
\end{theorem}

\begin{corollary} \label{cor:criterion-enough-proj}
In \Cref{cor:ab-env-from-functor-semisimple}, under the assumption (d), the abelian envelope even has enough projectives, and is given by the category $\rmod{\cS}$, for $\cS$ the category of global splitting objects in $\cC$.
\end{corollary}

\begin{proof}
    Under the assumptions in \Cref{cor:ab-env-from-functor-semisimple}, in particular, the pseudo-tensor category $\cC$ is pseudo-diagrammatic, so $\U(\cC)=\Uex(\cC)$, so by \Cref{thm:ab-env-projectives} it suffices to see that there is a non-zero global splitting object (for all morphism simultaneously). It was shown in the proof of \Cref{cor:ab-env-from-functor-semisimple} that this is the case under the assumption (d).
\end{proof}

To prove \Cref{thm:ab-env-projectives}, we start with an examination of projective objects in $\cC$ (see \Cref{definition:projective_obj}).
In the abelian case, the following observation is \Cref{ex:semisimple-splitting}(c):

\begin{lemma}\label{lemma:splitting_is_projective}
Every object $S \in \cS$ is a projective object in $\cC$.
\end{lemma}
\begin{proof}
If $S \in \cS$, then $S^{\ast} \in \cS$.
The functor $(- \otimes S^{\ast}): \cC \rightarrow \cC$ preserves cokernels and sends every morphism to a split morphism.
It follows that the functor
$$\cC(S,-) \simeq \cC( \one, - \otimes S^{\ast} )$$
sends cokernel sequences to split cokernel sequences.
\end{proof}

For the definition of weak (co)kernels, see \Cref{definition:weak_cokernel}.

\begin{lemma}\label{lemma:S_has_weak_kernels}
Let $S \in \cS$ and let $S \xrightarrow{\epsilon} \one$ be an epimorphism in $\cC$.
Then $\cS$ has weak kernels.
\end{lemma}
\begin{proof}
Let $\alpha: A \rightarrow B$ be a morphism in $\cS$.
Then $S \otimes \alpha: S \otimes A \rightarrow S \otimes B$ splits, and hence has a kernel $K$ which lies in $\cS$ as it is a direct summand of $S \otimes A$.
Let $\kappa: K \hookrightarrow S \otimes A$ denote the kernel embedding.
We claim that $K$ together with the morphism
$$
\kappa' := \left( K \xrightarrow{\kappa} S \otimes A \xrightarrow{\epsilon \otimes A} A \right)
$$
is a weak kernel of $\alpha$.
For this, we first compute
\begin{align*}
    \alpha \circ \kappa' = \alpha \circ (\epsilon \otimes A) \circ \kappa = (\epsilon \otimes B) \circ (S \otimes \alpha )\circ \kappa = 0.
\end{align*}
Next, let $T \xrightarrow{\tau} A$ be a morphism in $\cS$ such that $\alpha \circ \tau = 0$.
Since $T$ is a global splitting object, $\epsilon \otimes T$ is a split epi, and we denote by $\sigma: T \rightarrow S \otimes T$ a corresponding section.
We compute
\begin{equation}\label{equation:splitway}
    (\epsilon \otimes A) \circ (S \otimes \tau) \circ \sigma = \tau \circ (\epsilon \otimes T) \circ \sigma = \tau \circ \id = \tau.
\end{equation}
Last, let $\nu: S \otimes T \rightarrow K$ be the morphism induced by the universal property of $K$ as a kernel of $P \otimes \alpha$. We compute
$$
\kappa' \circ (\nu \circ \sigma) = 
(\epsilon \otimes A) \circ \kappa \circ \nu \circ \sigma = (\epsilon \otimes A) \circ (S \otimes \tau) \circ \sigma = \tau
$$
where the last equation is \Cref{equation:splitway}. Hence, $\tau$ factors via $\kappa'$. 
This gives the claim.
\end{proof}

We refer to the category of finitely presented functors (see \Cref{appendix:fp}) over $\cS$ by $\rmod{\cS}$.

\begin{corollary}\label{corollary:modS_abelian}
    If we have an epimorphism $S \xrightarrow{\epsilon} \one$ in $\cC$ with $S \in \cS$, then $\rmod{\cS}$ is abelian.
\end{corollary}
\begin{proof}
    This follows from \Cref{lemma:S_has_weak_kernels} and \Cref{theorem:Freyd_main}.
\end{proof}

From now on, we assume that we have a presentation $T \rightarrow S \rightarrow \one$ of the tensor unit in $\cC$ by global splitting objects $S,T \in \cS$.
In particular, $\rmod{\cS}$ is abelian under this assumption by \Cref{corollary:modS_abelian}.

\begin{corollary}\label{corollary:C_embeds_into_modS}
    The functor
    \begin{align*}
    \cC &\rightarrow \rmod{\cS} \\
    A &\mapsto \cC(-,A)|_{\cS}
    \end{align*}
    is well-defined and fully faithful.
\end{corollary}
\begin{proof}
Let $A \in \cC$ and let $T \rightarrow S \rightarrow \one$ be a presentation by global splitting objects. Then $T \otimes A \rightarrow S \otimes A \rightarrow A$ is a presentation of $A$ by global splitting objects.
By \Cref{lemma:splitting_is_projective}, $\cS \subseteq \cC$ is a full additive subcategory which consists of projective objects.
Now, the claim follows from \Cref{lemma:well-defined_restriction}.
\end{proof}

We get a monoidal structure on $\rmod{\cS}$ with tensor unit given by $\cC(-,\one)|_{\cS}$, and with a tensor product which is determined (up to natural isomorphism) by the following two properties:
\begin{enumerate}
    \item $\cS(-,S) \otimes \cS(-,T) \simeq \cS(-,S \otimes T)$ for all $S,T \in \cS$,
    \item $(M \otimes -)$ and $(- \otimes M)$ are right exact for all $M \in \rmod{\cS}$.
\end{enumerate}
The construction of this tensor product is explicitly given in \cite{BiesPos}.
Moreover, it follows from these two defining properties that the embedding of \Cref{corollary:C_embeds_into_modS}
$$
\cC \rightarrow \rmod{\cS}
$$
can be equipped with the structure of a strong monoidal functor.
In this sense, we view $\cC$ as a monoidal subcategory of $\rmod{\cS}$.

\begin{lemma}\label{lemma:tensoring_with_M}
Let $M \in \rmod{\cS}$ and let $S \in \cS$. Then $S \otimes M$ and $M \otimes S$ both lie in $\cS$.
\end{lemma}
\begin{proof}
Let $V \xrightarrow{u} U \rightarrow M$ be a presentation of $M$ in $\rmod{\cS}$ with $U,V \in \cS$.
Then 
$S \otimes V \xrightarrow{S \otimes u} S \otimes U \rightarrow S\otimes M \rightarrow 0$
is exact. Since $S \otimes u$ splits, it follows that $S\otimes M$ is a direct summand of $S \otimes U \in \cS$ and thus it lies itself in $\cS$.
The other case follows similarly.
\end{proof}

\begin{corollary}\label{corollary:to_split_exact_sequences}
    Tensoring with $S \in \cS$ sends short exact sequences in $\rmod{\cS}$ to split short exact sequences.
\end{corollary}
\begin{proof}
As $S$ is dualizable, tensoring with $S$ is exact.
By \Cref{lemma:tensoring_with_M}, tensoring with $S$ sends a short exact sequence to a short exact sequence consisting of projectives which therefore is split.
\end{proof}

\begin{corollary}\label{corollary:injectives_in_modS}
Every object $S \in \cS$ is injective in $\rmod{S}$.
\end{corollary}
\begin{proof}
Let $S \in \cS$. We have a natural isomorphism
$$\rmod{\cS}(-,S) \simeq \rmod{\cS}( - \otimes {}^*S,\one  ).$$
Now, the claim follows from \Cref{corollary:to_split_exact_sequences}.
\end{proof}

\begin{corollary}\label{corollary:equiv_weak_kernel_weak_cokernel}
A sequence $A \xrightarrow{\alpha} B \xrightarrow{\beta} C$ in $\cS$
is a weak kernel sequence if and only if it is a weak cokernel sequence.
\end{corollary}
\begin{proof}
Since $\cS$ is closed under taking duals, it suffices to show one direction.
So, let $\alpha$ be a weak kernel of $\beta$.
Then the sequence 
$A \xrightarrow{\alpha} B \xrightarrow{\beta} C$
is exact in $\rmod{\cS}$.
Let $S \in \cS$. Since $S$ is injective in $\rmod{S}$ by \Cref{corollary:injectives_in_modS}, the sequence
$$\rmod{\cS}(A,S) \xleftarrow{\rmod{\cS}(\alpha,S)} \rmod{\cS}(B,S) \xleftarrow{\rmod{\cS}(\beta,S)} \rmod{\cS}(C,S)$$
is exact and hence $\beta$ is a weak cokernel of $\alpha$.
\end{proof}

\begin{remark}
For every $M \in \rmod{\cS}$, the functor $(- \otimes M): \rmod{\cS} \rightarrow \rmod{\cS}$ has a right adjoint $[M,-]$ defined by
$$
[M,-] := \ker\big( ({}^*S \otimes -) \xrightarrow{({}^*\rho \otimes -)} ({}^*T \otimes -) \big)
$$
where $T \xrightarrow{\rho} S \rightarrow M$ is a presentation by global splitting objects, see \cite{BiesPos}*{Theorem 3.3.8}. Similarly, $(M \otimes -)$ has a right adjoint.
We note that these right adjoints send global splitting objects to global splitting objects: indeed, let $U \in \cS$, then $({}^*S \otimes U) \xrightarrow{({}^*\rho \otimes U)} ({}^*T \otimes U)$ splits, and hence its kernel is a direct summand of $({}^*S \otimes U)$ and thus lies in $\cS$.
\end{remark}

\begin{lemma}\label{lemma:tensor_product_exact_modS}
The tensor product of $\rmod{\cS}$ is exact.
\end{lemma}
\begin{proof}
We prove that the criterion given in \Cref{theorem:Freyd_main} is satisfied.
Let $A \xrightarrow{\alpha} B \xrightarrow{\beta} C$ be a weak kernel sequence in $\cS$.
Let $M \in \rmod{\cS}$.
We need to show that the sequence
\begin{equation}\label{equation:exactness_proof}
  M \otimes A \xrightarrow{M \otimes \alpha} M \otimes B \xrightarrow{M \otimes \beta} M \otimes C  
\end{equation}
is exact in $\rmod{\cS}$. Note that all morphisms of \eqref{equation:exactness_proof} actually lie in $\cS$ by \Cref{lemma:tensoring_with_M}.
Thus, we need to show that \eqref{equation:exactness_proof} is a weak kernel sequence. By \Cref{corollary:equiv_weak_kernel_weak_cokernel}, this is equivalent to \eqref{equation:exactness_proof} being a weak cokernel sequence. But this is indeed true: for all $S \in \cS$, the sequence
  $$
  \cS(M \otimes A,S) \leftarrow \cS(M \otimes B,S) \leftarrow \cS(M \otimes C,S)
  $$
  is isomorphic to the sequence
  $$
  \cS(A,[M,S]) \leftarrow \cS(B,[M,S]) \leftarrow \cS(C,[M,S])
  $$
which is exact since $A \xrightarrow{\alpha} B \xrightarrow{\beta} C$ is also a weak cokernel sequence by \Cref{corollary:equiv_weak_kernel_weak_cokernel}.
Exactness of $(- \otimes M)$ is proven similarly.
\end{proof}

\begin{lemma}\label{lemma:dualizable_criterion}
Let $\cA$ be a monoidal abelian category with an exact tensor product.
Let 
$$A \rightarrow B \rightarrow C \rightarrow 0$$
be an exact sequence such that $A$ and $B$ have left (right, resp.) duals.
Then $C$ has a left (right, resp.) dual.
\end{lemma}
\begin{proof}
The proof can be taken from \cite{BEO}*{Proposition 2.36}, since that proof only makes use of the abelian structure and the exactness of the tensor product.
\end{proof}

\begin{theorem}
The category $\rmod{\cS}$ is a tensor category.
\end{theorem}
\begin{proof}
It is abelian (\Cref{corollary:modS_abelian}) and we have $\End_{\rmod{\cS}}( \one ) \cong \End_{\cC}( \one ) \cong \kk$ (\Cref{corollary:C_embeds_into_modS}).
Moreover, every object in $\rmod{\cS}$ arises as the cokernel of a morphism between dualizable objects. Since the tensor product is exact (\Cref{lemma:tensor_product_exact_modS}), we can apply \Cref{lemma:dualizable_criterion} and see that $\rmod{\cS}$ is rigid.
\end{proof}

\begin{lemma}\label{lemma:aux_lemma}
Let $F: \cC \rightarrow \cT$ be a monoidal faithful functor into a tensor category $\cT$.
Then $F$ preserves cokernels.
\end{lemma}
\begin{proof}
Let $A \rightarrow B \rightarrow C$ be a cokernel sequence in $\cC$.
Let $S \in \cS$.
Then $S \otimes A \rightarrow S \otimes B \rightarrow S \otimes C$ is a cokernel sequence and split in $\cS$.
Hence, $F(S \otimes A) \rightarrow F(S \otimes B) \rightarrow F(S \otimes C)$ and consequently $F(S) \otimes F(A) \rightarrow F(S) \otimes F(B) \rightarrow F(S) \otimes F(C)$ are cokernel sequences. Since $F$ is faithful, $F(S) \not\simeq 0$.
As $\cT$ is a tensor category, $F(S) \otimes -$ is a faithful exact functor (see, e.g., \cite{CEOP}*{Lemma 1.2.5.}), and hence reflects exactness.
It follows that $F(A) \rightarrow F(B) \rightarrow F(C)$ is a cokernel sequence.
\end{proof}

We can now prove the main theorem of this section.

\begin{proof}[Proof of \Cref{thm:ab-env-projectives}]

(b) $\Rightarrow$ (a): Assume that the tensor unit of $\cC$ has a presentation by global splitting objects.
Let $\cT$ be a tensor category over $\kk$.
We need to prove that restriction along $\cC \rightarrow \rmod{S}$ induces an equivalence of categories
$$
\Ten^{\text{exact}}(\rmod{\cS},\cT)\to\Ten^{\text{faith}}(\cC,\cT).
$$
The functor is fully faithful since the natural transformations are all uniquely determined by their components at objects in $\cS$.
We show essential surjectivity:
Let $F: \cC \rightarrow \cT$ be a monoidal faithful functor.
We can restrict this functor along $\cS \rightarrow \cC$ and then use the universal property of $\rmod{S}$ in order to obtain a right exact functor $F': \rmod{S} \rightarrow \cT$. It is easy to see that the monoidal structure of $F$ lifts to a monoidal structure of $F'$.
Since $F'$ is monoidal and right exact, it is also left exact and hence exact.
Last, the restriction of $F'$ to $\cC$ gives back $F$ by \Cref{lemma:aux_lemma}.

(a) $\Rightarrow$ (c): Assume that $\cC$ has a monoidal abelian envelope $\cT$ with enough projectives. Then, in particular, $\cC$ embeds fully faithfully into a tensor category, so $\U(\cC)=\Uex(\cC)$ by \cite{CEOP}*{Cor.~2.3.3(2)}.
Let $P$ be a non-zero projective object in $\cT$. Since $\cT$ is a tensor category, projective objects and injective objects coincide.
By \cite{CEOP}*{Lemma~2.2.4}, any object of $\cT$ is a subquotient of an object from $\cC$, in particular, any projective object in $\cT$ is a direct summand of an object in $\cC$ and therefore already lies in $\cC$.
It follows that $P$ already lies in $\cC$.
Lastly, any projective object in $\cT$ is a global splitting object in $\cC$ by part (b) of \Cref{ex:semisimple-splitting}. Thus, $\cC$ has a non-zero global splitting object $P$.

(c) $\Rightarrow$ (b): Let $S$ be a non-zero global splitting object, then
$$
S\o S^*\o S\o S^*\xrightarrow{\ev\o\id-\id\o\ev} S\o S^* \xrightarrow{\ev}\one \to 0
$$
is exact due to the assumption $\U(\cC)=\Uex(\cC)$, so there is a presentation of $\one$ by global splitting objects.
\end{proof}

\section{Applications}

\subsection{Categories formed by partition diagrams} Combining some of the above results, we obtain the following criteria for the existence of abelian envelopes for certain subcategories of $\cS_t=\RepSt$, for $t\in \Bbbk$. 

\begin{theorem} \label{thm:partition-cats} Let $\cC$ be any pseudo-tensor subcategory of $\cS_t=\RepSt$ containing the object $[1]$ such that the hom-spaces $\cC([m],[n])$ are spanned by $P_{m,n}\cap\cC([m],[n])$, for all $m,n\ge0$. Assume
\begin{enumerate}
    \item[(a)] the functor $\cS_t(-,\one)|_\cC$ is representable, or
\item [(b)] all the functors $\cS_t(-,\one)|_{\cC_i}$ are representable, for $i\ge0$, where $\cC_i:=\langle [0],[1],\dots,[i]\rangle\subset\cC$ is the pseudo-abelian subcategory generated by the objects $[0],[1],\dots,[i]$.
\end{enumerate}
Then the natural faithful functor $\cC\to\cS_t$ (case (a)) or its cocontinuous extension to the ind-completions (case (b)) has a right adjoint, and if $\kk$ is of characteristic $0$, then $\cC$ has an abelian envelope with the quotient property. 

Moreover, in case (a), the abelian envelope has enough projectives and is realized as $\rmod\cS$ from \Cref{appendix:fp}, where $\cS$ is full subcategory of splitting objects in $\cC$.
\end{theorem}

\begin{proof} Let $\iota$ be the natural faithful functor $\cC\to\cS_t$. It is a dominant linear monoidal functor by our assumptions. $\cC$ is pseudo-diagrammatic by \Cref{prop:diag-for-cob}.

In case (a), $\iota$ has a right adjoint by \Cref{corollary:adjoints_rigid_case}. Now $\cS_t$ admits a dominant linear monoidal functor $F$ to a semisimple tensor category such that $F$ admits a right adjoint. Namely, if $t\not\in\NN_0$, $F$ can be taken to be the identity, and if $t\in\NN_0$, then $F$ can be taken to be the restriction
$$
\cS_t \to \cS_{-1} \boxtimes \Rep(S_{t+1}) ,
$$
as discussed in \cite{complex-rank-1}*{Section~2.3}, \cite{HK}*{Corollary/Definition~1.4.5}, or \cite{FHL}*{Section~2.6}. Note that the tensor product of semisimple categories on the right is indeed semisimple, as $\kk$ is a splitting field for both factors, see \cite{FLP}*{Theorem~5.11, Lemma~7.25}. Hence, $F\iota$ is a dominant linear monoidal functor with a right adjoint. Then the assertion follows from \Cref{cor:ab-env-from-functor-semisimple}. The assumption that $\kk$ has characteristic $0$ is used for the existence of the adjoint functors for the above restriction functors and for the semisimplicity of $\Rep(S_{t+1})$. 

In case (b), $\Ind(\iota)$ has a right adjoint by \Cref{lem:adj-filtered}. Using a functor $F$ as before, we obtain a dominant linear monoidal functor $\iota F$ to a semisimple tensor category such that $\Ind(\iota F)$ has a right adjoint. Again the assertion follows from \Cref{cor:ab-env-from-functor-semisimple}.

The last assertion follows from \Cref{cor:criterion-enough-proj}.
\end{proof}

As a first application, we obtain a new proof for the existence of abelian envelopes for $\cS_t$. Candidate categories for the abelian envelopes in this case were constructed already in \cite{Del}, and the first proof of the existence of such abelian envelopes was given in \cite{CO-ab}.

\begin{corollary} \label{cor:St}
The categories $(\cS_t)_{t\in\kk}$ have abelian envelopes with enough projectives, for any field $\kk$ of characteristic $0$.
\end{corollary}

\begin{proof} This follows immediately from \Cref{thm:partition-cats} with $\cC=\cS_t$.
\end{proof}

We will demonstrate the criterion from \Cref{thm:partition-cats} in more cases in \Cref{sec:hyperoct} and \Cref{sec:modified-St}.

\subsection{Interpolation categories for hyperoctahedral groups} \label{sec:hyperoct}

Assume $\kk$ is of characteristic $0$. Recall from \Cref{bg:RepSt} the definition of the categories $\cH_t$ as the subcategories of $\cS_t$ corresponding to partition diagrams with only even components. 

For all $m,j\ge0$ and any partition diagram $f\in P_{m,j}$, we define recursively (as in \cite{CO-blocks}*{Equation~(2.1)})
$$
x(f) := f - \sum_{f'<f} x(f') \quad\in \cS_t([m],[j]),
$$
where $f'<f$ if $f$ is a proper coarsening of $f$. Note that if $f$ has only even components, then any coarsening of $f$ will have only such components, too, so $x(f)$ will be a well-defined morphism even in $\cH_t$. 
We set
$$
x_j := x(\id_{[j]}) = x(\underbrace{\bar\dots\bar}_{j}) \quad\in \cH_t([j],[j])
$$

The morphisms $x_j$ have the following absorption property.

\begin{lemma}[\cite{Flake-Laugwitz}*{Claim in Proof of Lemma~3.13}] \label{lem:cancelation-x} For any $m,j\ge0$, let $g\in P_{m,j}$ be a partition diagram with a component containing two lower points. Then $x_j g=0$.
\end{lemma}

It follows immediately that $x_j$ is an idempotent endomorphism, for all $j\ge0$, as $x_j f=0$ for any proper coarsening $f$ of $\id_{[j]}$.

For any $j\ge0$, we also set
$$
e_j := \frac{1}{j!} \sum_{\s\in S_j} \s
\quad\in \cH_t([j],[j]) ,
$$
where elements of the symmetric group are viewed as permutation diagrams, which defines a group homomorphism from $S_j$ to $\Aut_{\cS_t}([j])$.
Then $x_j$ and $e_j$ are commuting idempotent endomorphisms, as $\sigma x(\id_{[j]})=x(\sigma)=x(\id_{[j]})\sigma$, for all $\sigma\in S_j$, so $x_j e_j$ is an idempotent endomorphism.

For $m,j\ge0$ and any partition diagram $f\in P_{m,j}$, we define recursively
$$
x'(f) := f - \sum_{f'<'f} x'(f') \quad\in \cS_t([m],[j]),
$$
where $f'<'f$ if and only if $f'$ is a proper coarsening of $f$ such that all upper components of $f$ are preserved in $f'$.

For any $j\ge0$, we define $p_j$ as the unique partition diagram with $j$ upper points, no lower points, and $j$ components. 

\begin{lemma} \label{lem:computation-H} For any $m,j\ge0$, let $g\in P_{m,j}$ be a partition diagram every component of which has an even number of points and at most one lower point. Let $f\in P_{m,0}$ be the partition diagram induced on the $m$ upper points. Then
$$
p_j x_j e_j g = x'(f)
$$
in $\cS_t$.
\end{lemma}

\begin{proof} We compute $x_j g = x'(g)$, as $x_j$ is a linear combination of coarsening of the identity, post-composing with which produces coarsenings of $g$ that involve exactly those components that have lower points. We also compute $p_j e_j=p_j$, as $p_j\sigma=\sigma$ for all $\sigma\in S_j$. Hence, the left-hand side of the above asserted identity equals
$$
p_j x'(g)
$$
which is $x'(f)$, as post-composition with $p_j$ has just the effect of removing all lower points from the components that contain them. 
\end{proof}

We set $\cC_i:=\langle [0],[1],\dots,[i]\rangle$, the pseudo-abelian subcategory of $\cH_t$ generated by the objects $[0],[1],\dots,[i]$. We define $X_j$ as the object in $\cH_t$ corresponding to the idempotent endomorphism $x_j e_j$, for all $j\ge0$.

\begin{proposition} \label{lem:main-H} For every $i\ge0$, the functor $\cS_t(-,\one)|_{\cC_i}$ is representable by $X:=\bigoplus_{0\le j\le i} X_j$.
\end{proposition}

\begin{proof} Set $p:=(p_j x_j e_j)_{0\le j\le i}\: X=\bigoplus_{0\le j\le i}X_j\to\one$. It suffices to show that for all $0\le m\le i$, the linear map
$$
\varphi:=(p\circ-)\: \cH_t([m],X) \to \cS_t([m],\one)
$$
is bijective.

For all $0\le j\le i$, let the symmetric group $S_j$ act on the set $P_{m,j}$ of partition diagrams with $m$ upper and $j$ lower points by permuting the lower points. This action preserves the subset $P'_{m,j}$ of partition diagrams such that every component has an even number of points, and even the subset $P''_{m,j}$ of $P'_{m,j}$ of partition diagrams such that every component has at most one lower point. Let $Q_{m,j}$ be a set of orbit representatives of $S_j$ acting on $P''_{m,j}$. We can alternatively think of $Q_{m,j}$ as the set of set partitions of the set $\{1,\dots,m\}$ with exactly $j$ odd components.

Let $q_j\: X_j\to X$ be the embedding of $X_j$ into the direct sum $X$.
Consider the following collections of morphisms:
\begin{align*}
B := \big( q_j x_j e_j g : 0\le j\le i, g \in Q_{m,j} \Big) 
&\quad \subset M:= \cH_t([m],X) , 
\\
B' := \Big( x'(f) \: f\in P_{m,0} \Big) 		%
&\quad \subset M':= \cS_t([m],\one) .
\end{align*}

We claim that $B$ is a spanning set of $M$. Indeed, by definition, a spanning set of $M$ is given by
$$
\big( q_j x_j e_j g : 0\le j\le i, g \in P'_{m,j} \Big) .
$$
Now if some $g$ in this collection has a component with two lower points, then $x_j g$ is zero by \Cref{lem:cancelation-x}. Hence, it suffices to pick $g$ in $P''_{m,j}$. Furthermore, if two elements $g,g'$ from $P''_{m,j}$ are related to each other by a permutation of lower points, then $e_j g=e_j g'$, so it suffices to pick $g$ in $Q_{m,j}$.

We also claim that $B'$ is a basis of $M'$. Consider $M'$ as a filtered vector space, where every partition diagram is assigned as degree its number of components. Then up to a term of lower degree, $x'(f)$ agrees with $f$, for all $f\in P_{m,0}$, and the morphisms corresponding to the partition diagrams $P_{m,0}$ form a basis of $M'$ by definition. This shows the claim.

Now using \Cref{lem:computation-H}, we compute
$$
\varphi( q_j x_j e_j g )
= p_j x_j e_j g
= x'(f),
$$
for all $0\le j\le i$ and $g\in Q_{m,j}$, where $f$ is the partition diagram induced by $g$ on its $m$ upper points. In this situation, $g\in\bigsqcup_{0\le j\le i} Q_{m,j}$ is determined by the partition $f$ it induces, and any $f\in P_{m,0}$ is induced by some $g\in \bigsqcup_{0\le j\le i} Q_j$, namely the one whose components are the ones of $f$, modified by adding exactly one lower point to each odd component.

Hence, the elements of $B$ are sent to the the elements of $B'$ via the linear map $\varphi$, up to terms of lower degrees. This implies that $\varphi$ is bijective.
\end{proof}

\begin{corollary} \label{cor:main-H} The cocontinuous extension to ind-completions of the natural faithful functor $\cH_t\to\cS_t$ has a right adjoint, and $\cH_t$ has a monoidal abelian envelope with the quotient property for each $t\in\kk$.
\end{corollary}

\begin{proof} This follows from \Cref{lem:main-H} using \Cref{thm:partition-cats}.
\end{proof}

\begin{remark} In the upcoming paper \cite{FLP-future}, we show that the abelian envelope of $\cH_t$ even has enough projectives.   
\end{remark}

\subsection{Interpolation categories for modified symmetric groups} \label{sec:modified-St} Let $\kk$ be any field. Recall from \Cref{bg:RepSt} the definition of the subcategories $\cS'_t$ of $\cS_t$ corresponding to partitions diagrams with an even number of odd components.

Assume for the rest of this subsection that $t\neq0$ and set 
$$
e_0 := \id_{[0]}, \quad
e_1 := t^{-1} \tp{1,0,11} ,
$$
then $e_j$ can be viewed as an idempotent endomorphism in $\cS'_t([j],[j])$ or $\cS_t([j],[j])$, for $j=0,1$. Let $p_j$ be the unique partition diagram in $P_{j,0}$, for $j=0,1$. Let $X_j$ be the object in $\cS'_t$ or $\cS_t$ corresponding to the idempotent $e_j$, then $p_j$ represents a morphism from $X_j$ to $\one=[0]$ in $\cS_t$ (but not in $\cS'_t$ for $j=1$). Note that $p_j e_j=p_j$ for $j=0,1$.

\begin{proposition} \label{prop:main-S-prime}
    The functor $\cS_t(-,\one)|_{\cS'_t}$ is representable by $X:=X_0\oplus X_1$.
\end{proposition}

\begin{proof} Similar to \Cref{lem:main-H}, it suffices to show that the linear map
$$
\varphi:=(p\circ-)\:\cS'_t([m],X) \to \cS_t([m],\one)
$$
is bijective for all $m\ge0$, with $p:=(p_0,p_1)\:X=X_0\oplus X_1\to\one$.

Let $P'_{m,j}\subset P_{m,j}$ be those partition diagrams that have an even number of odd components, and let $P''_{m,j}\subset P'_{m,j}$ be those partition diagrams all of whose lower points are the only points in their respective component, for $j=0,1$. We can alternatively think of $P''_{m,0}$ as the set of set partitions of $\{1,\dots,m\}$ with an even number of odd components, and of $P''_{m,1}$ as the set of set partitions of $\{1,\dots,m\}$ with an odd number of odd components.

Let $q_j\:X_j\to X$ be the embeddings into the direct sum for $j=0,1$.

Consider the following collections of morphisms:
\begin{align*}
B:=\Big( q_j e_j g : 0\le j\le1, g\in P''_{m,j} \Big) 
\quad\subset M:=\cS'_t([m],X)
\\
B':=\Big( f\in P_{m,0} \Big)
\quad\subset M':=\cS_t([m],\one) .
\end{align*}
We claim that $B$ is a spanning set of $M$. Indeed, a spanning set of $M$ is given by the morphisms $q_j e_j g$, for $0\le i\le 1$ and $g\in P'_{m,j}$, but if $g\in P'_{m,j}$ has a lower point that is not the only point in its component, then $j=1$ and $e_1 g=t^{-1} e_1 g'$ for the unique $g'\in P''_{m,1}$ that is obtained from $g$ by placing the unique lower point in its own component. This shows the claim. 

We also note that $B'$ is a basis of $M'$ by definition, and that $\varphi$ induces a bijection between $B$ and $B'$, as $p_j e_j g = f$ for all $j=0,1$, $g\in P''_{m,j}$, for the unique partition $f$ induced by $g$ on its $m$ upper points. This proves the assertion.
\end{proof}

\begin{corollary} \label{cor:main-S-prime} The natural faithful functor $\cS'_t\to\cS_t$ has a right adjoint, for each $t\in\kk\setminus\{0\}$. Moreover, if $\kk$ has characteristic $0$, then $\cS'_t$ has an abelian envelope with enough projectives that can be realized as $\rmod\cS$ for the category $\cS$ of splitting objects in $\cS'_t$, for each $t\in\kk\setminus\{0\}$.
\end{corollary}

\begin{proof} This follows from \Cref{prop:main-S-prime} using \Cref{thm:partition-cats}.
\end{proof}

\appendix

\section{}

\subsection{Free cocompletions and ind-completions} \label{appendix:ind}

Let $\cC$ be a small category.
Let $\cSet$ denote the category of sets.
The category of presheaves $\PSh( \cC )$ on $\cC$ is the category whose objects are functors of type $\cC{\op} \rightarrow \cSet$ and whose morphisms are natural transformations.
We denote the \emph{Yoneda embedding} as follows:
\begin{align*}
    \cC \xrightarrow{Y} \PSh( \cC ), \qquad
    A \mapsto \cC( -, A ) .
\end{align*}
A category is called \emph{cocomplete} if it admits all colimits of small diagrams.
A functor is called \emph{cocontinuous} if it commutes with all colimits of small diagrams.
The category of presheaves can be seen as the \emph{free cocompletion of $\cC$} in the following sense: for every functor of type 
$\cC \xrightarrow{G} \cD$
into a cocomplete category $\cD$, there exists a unique (up to natural isomorphism) cocontinuous functor 
$\PSh(\cC) \xrightarrow{\widehat{G}} \cD$
such that the following diagram commutes (up to natural isomorphism):
\begin{center}
       \begin{tikzpicture}
        \coordinate (r) at (4,0);
        \coordinate (d) at (0,-1.5);
        \node (A) {$\cC$};
        \node (B) at ($(A)+(r)$) {$\PSh( \cC )$};
        \node (C) at ($(B) + (d)$) {$\cD$};
        \draw[->,thick] (A) to node[above]{$Y$}(B);
        \draw[->,thick] (A) to node[below]{$G$} (C);
        \draw[->,thick,dashed] (B) to node[right=3pt]{$\widehat{G}$} (C);
        \end{tikzpicture}
\end{center}

\begin{lemma}
Let $G: \cC \rightarrow \cD$ be a functor whose target is a cocomplete category.
The functor $\widehat{G}\:\PSh(\cC)\to\cD$ always has a right adjoint $R$
which is given by
\begin{align*}
    X &\mapsto \cD( G(-), X )|_{\cC\op}
\end{align*}
\end{lemma}

If $G: \cC \rightarrow \cD$ is an arbitrary functor (with $\cD$ an arbitrary category), then the universal property of $\PSh( \cC )$ applied to the composite functor
\[
\cC \xrightarrow{G} \cD \xrightarrow{Y} \PSh( \cD )
\]
defines a functor
\[
\PSh( G ): \PSh( \cC ) \rightarrow \PSh( \cD ).
\]

A small category $\cI$ is called \emph{filtered} if
\begin{itemize}
    \item $\cI \neq \emptyset$,
    \item for all $i,j \in \cI$, there exists a $k \in \cI$ and morphisms $i \rightarrow k$, $j \rightarrow k$,
    \item for all parallel morphisms $\alpha, \beta: i \rightarrow j \in \cI$, there exists a morphism $\gamma: j \rightarrow k$ such that $\gamma \circ \alpha = \gamma \circ \beta$.
\end{itemize}
A \emph{filtered colimit} is a colimit of a functor whose source is a small filtered category.
The \emph{ind-completion} of a small category $\cC$ is defined as the full subcategory $\Ind( \cC ) \subseteq \PSh( \cC )$ spanned by those presheaves which arise as a filtered colimit of representable presheaves. We note that we can write this embedding as
\begin{align*}
\Ind( \cC ) &\rightarrow \PSh( \cC )\\
\ti{X} &\mapsto \Ind( \cC )( -, \ti{X} )|_{\cC\op} .
\end{align*}
Moreover, the Yoneda embedding factors over $\Ind( \cC )$:
\[
\cC \xrightarrow{Y} \Ind( \cC ).
\]
The ind-completion has the following universal property: for every functor of type 
$\cC \xrightarrow{G} \cD$
into a category $\cD$ with filtered colimits, there exists a unique (up to natural isomorphism) functor which commutes with filtered colimits
$\Ind(\cC) \xrightarrow{\ti{G}} \cD$
such that the following diagram commutes (up to natural isomorphism):
\begin{center}
       \begin{tikzpicture}
        \coordinate (r) at (4,0);
        \coordinate (d) at (0,-1.5);
        \node (A) {$\cC$};
        \node (B) at ($(A)+(r)$) {$\Ind( \cC )$};
        \node (C) at ($(B) + (d)$) {$\cD$};
        \draw[->,thick] (A) to node[above]{$Y$}(B);
        \draw[->,thick] (A) to node[below]{$G$} (C);
        \draw[->,thick,dashed] (B) to node[right=3pt]{$\ti{G}$} (C);
        \end{tikzpicture}
\end{center}

If $G: \cC \rightarrow \cD$ is an arbitrary functor (with $\cD$ an arbitrary category), then the universal property of $\Ind( \cC )$ applied to the composite functor
\[
\cC \xrightarrow{G} \cD \xrightarrow{Y} \Ind( \cD )
\]
defines a functor
\[
\Ind( G ): \Ind( \cC ) \rightarrow \Ind( \cD ).
\]
We obtain the following diagram of functors which commutes up to natural isomorphism:
\begin{center}
  \begin{tikzpicture}
    \coordinate (r) at (3.5,0);
    \coordinate (d) at (0,-1.5);
    
    \node (11) {$\cC$};
    \node (12) at ($(11)+(r)$) {$\cD$};

    \node (21) at ($(11) + (d)$){$\Ind( \cC )$};
    \node (22) at ($(21)+(r)$) {$\Ind( \cD )$};

    \node (31) at ($(21) + (d)$){$\PSh( \cC )$};
    \node (32) at ($(31)+(r)$) {$\PSh( \cD )$};
    
    \draw[->,thick] (11) to node[above]{$G$} (12);
    \draw[->,thick] (21) to node[above]{$\Ind( G )$}(22);
    \draw[->,thick] (31) to node[above]{$\PSh( G )$} (32);

    \draw[->,thick] (11) to (21);
    \draw[->,thick] (12) to (22);
    
    \draw[->,thick] (21) to (31);
    \draw[->,thick] (22) to  (32);
  \end{tikzpicture}
\end{center}
The vertical functors are all full and faithful, and we will treat them as full inclusions of subcategories in our notation whenever this is convenient.

\begin{lemma} \label{lem:adj-ind}
Let $G: \cC \rightarrow \cD$ be a functor.
The functor $\Ind( G ): \Ind( \cC ) \rightarrow \Ind( \cD )$ has a right adjoint if and only if for every $X \in \cD$, the presheaf $\cD( G(-), X )|_{\cC\op}$ lies in $\Ind( \cC )$, i.e., there exists a $\ti Y \in \Ind( \cC )$ and an isomorphism in $\PSh( \cC )$:
\[
\cD( G(-), X )|_{\cC\op} \cong \Ind( \cC )( -, \ti{Y} )|_{\cC\op}.
\]
\end{lemma}
\begin{proof}
Recall that $\PSh( G )$ always has a right adjoint 
\[
R: \widehat{X} \mapsto \cD( G(-), \widehat{X} )|_{\cC\op}.
\]
Assume that $\Ind( G )$ has a right adjoint $R'$. Then for all $\ti X \in \Ind(\cD)$:
\[
R \ti X \cong \cD( G(-), \ti X )|_{\cC\op} \cong \Ind( \cC )( -, R'\ti X )|_{\cC\op} \cong R'\ti X.
\]
In particular, we see that if $R'$ exists, it necessarily is given by a factorization of $R$ via the full subcategory $\Ind( \cC ) \subseteq \PSh( \cC )$.

Conversely, assume that for every $X \in \cD$, there exists a $\ti Y \in \Ind( \cC )$ and an isomorphism in $\PSh( \cC )$:
\[
\cD( G(-), X )|_{\cC\op} \cong \Ind( \cC )( -, \ti{Y} )|_{\cC\op} \cong \ti Y
\]
Let $\ti X = \colim_{i} X_i$ be an object in $\Ind( \cD )$ given by a filtered colimit of objects $X_i \in \cD$.
The natural map
\[
\colim_i \cD( GA, X_i ) \rightarrow \Ind( \cD )( GA, \colim_i X_i )
\]
is an isomorphism for all $A \in \cC$, since $GA \in \cD$ is compact in $\Ind( \cD )$.
By assumption, each $\cD( G-, X_i )|_{\cC\op}$ lies in $\Ind( \cC )$.
Thus, $\colim_i \cD( G-, X_i )|_{\cC\op}$ lies in $\Ind( \cC )$, since $\Ind( \cC )$ is closed under filtered colimits and the inclusion $\Ind( \cC ) \subseteq \PSh( \cC )$ commutes with filtered colimits. Thus, 
\[
R( \ti X ) = \Ind( \cD )( G(-), \ti X  )|_{\cC\op}
\]
lies in $\Ind( \cC )$, which proves that $R$ factors over $\Ind( \cC )$.
\end{proof}

\subsection{Finitely presented functors}\label{appendix:fp}

The theory of finitely presented functors goes back to Auslander \cite{Aus66}.
Let $\cA$ be an additive category.
We denote by $\cAb$ the category of abelian groups.
A functor $M$ of type $\cA^{\op} \rightarrow \cAb$ is called \emph{finitely presented} if it is the cokernel of a morphism between representable functors, i.e., if there exists a morphism $A \xrightarrow{\alpha} B$ in $\cA$ and an exact sequence of functors\footnote{This is a sequence which is exact for all evaluations at objects in $\cA$.}
\begin{center}
  \begin{tikzpicture}[mylabel/.style={fill=white}]
      \coordinate (r) at (4,0);
      
      \node (A) {$\cA(-,B)$};
      \node (B) at ($(A)+(r)$) {$\cA(-,A)$};
      \node (C) at ($(B) + 0.5*(r)$) {$M$};
      \node (D) at ($(C) + 0.5*(r)$) {$0$.};
      
      \draw[->,thick] (A) to node[above]{$\cC(-,\alpha)$} (B);
      \draw[->,thick] (B) to (C);
      \draw[->,thick] (C) to (D);
  \end{tikzpicture}
\end{center}
We denote by $\rmod{\cA}$ the category whose objects are finitely presented functors and whose morphisms are given by natural transformations.
The category $\rmod{\cA}$ is additive and has cokernels, which are given pointwise.
It is the cokernel completion of $\cA$ in the following sense:
for every functor of type 
$\cA^{\op} \rightarrow \cB$
into an additive category $\cB$ with cokernels, there exists a unique (up to natural isomorphism) \emph{right exact} functor $\rmod{\cA} \rightarrow \cB$ (i.e., it preserves cokernels) such that the following diagram commutes (up to natural isomorphism):
\begin{center}
       \begin{tikzpicture}
        \coordinate (r) at (4,0);
        \coordinate (d) at (0,-1.5);
        \node (A) {$\cA$};
        \node (B) at ($(A)+(r)$) {$\rmod{\cA}$};
        \node (C) at ($(B) + (d)$) {$\cB$};
        \draw[->,thick] (A) to node[above]{$Y$}(B);
        \draw[->,thick] (A) to (C);
        \draw[->,thick,dashed] (B) to (C);
        \end{tikzpicture}
\end{center}
Here, $Y$ denotes the Yoneda embedding $A \mapsto \cA(-,A)$.

It is convenient to use the concept of a projective object not only in an abelian category but in an additive category.

\begin{definition}\label{definition:projective_obj}
    An object $P$ in an additive category $\cA$ is called \emph{projective} if $\cA(P,-)$ preserves cokernels.
    We say that $\cA$ \emph{has enough projectives} if for every object $A \in \cA$, we have projective objects $P,Q$ and a sequence
    \begin{center}
  \begin{tikzpicture}[mylabel/.style={fill=white}]
      \coordinate (r) at (3,0);
      
      \node (A) {$Q$};
      \node (B) at ($(A)+(r)$) {$P$};
      \node (C) at ($(B) + (r)$) {$A$};
      \node (D) at ($(C) + (r)$) {$0$};
      
      \draw[->,thick] (A) to (B);
      \draw[->,thick] (B) to (C);
      \draw[->,thick] (C) to (D);
  \end{tikzpicture}
\end{center}
where $P \rightarrow A$ is the cokernel of $Q \rightarrow P$.
We call such a sequence a \emph{presentation of $A$ (by projectives)}.
\end{definition}

\begin{remark}\label{remark:representables_are_projective}
It follows from the Yoneda lemma that representable functors are projective objects in $\rmod{\cA}$.
In particular, $\rmod{\cA}$ has enough projectives.    
\end{remark}

\begin{remark}\label{remark:commutative_square_calculus}
    Let $Q \xrightarrow{\rho} P \rightarrow A$ and $Q' \xrightarrow{\rho'} P' \rightarrow B$ be presentations of objects $A,B \in \cA$ by projectives.
    Let $\mathcal{R}$ be the abelian group of pairs of morphisms $(\alpha:P \rightarrow P', \omega: Q \rightarrow Q')$ such that the following square commutes:
    \begin{center}
  \begin{tikzpicture}[mylabel/.style={fill=white}]
    \coordinate (r) at (3,0);
    \coordinate (d) at (0,-1.2);
    
    \node (11) {$Q$};
    \node (12) at ($(11)+(r)$) {$P$};

    \node (21) at ($(11) + (d)$){$Q'$};
    \node (22) at ($(21)+(r)$) {$P'$};
    
    \draw[->,thick] (11) to node[above]{$\rho$}(12);

    \draw[->,thick] (21) to node[above]{$\rho'$}(22);

    \draw[->,thick] (11) to node[left]{$\omega$}(21);
    \draw[->,thick] (12) to node[right]{$\alpha$}(22);
  \end{tikzpicture}
\end{center}
Let $\mathcal{R}' \subseteq \mathcal{R}$ be the subgroup given by those pairs $(\alpha, \omega)$ such that $\alpha$ factors via $\rho'$.
Then it is a standard computation that we have an isomorphism
\[
\mathcal{R}/\mathcal{R}' \simeq \Hom_{\cA}(A,B)
\]
which sends a commutative square of the above form to its induced morphism between cokernels. 
\end{remark}

As can be seen from \Cref{remark:commutative_square_calculus}, the category $\rmod{\cA}$ can be viewed as the \emph{cokernel completion} or \emph{category of presentations} of $\cA$.

\begin{lemma}\label{lemma:well-defined_restriction}
Let $\cP \subseteq \cA$ be a full additive subcategory which consists of projective objects.
If every object in $\cA$ admits a presentation by objects in $\cP$, then restriction to $\cP$ gives a well-defined functor
$$
(M \mapsto M|_{\cP}): \rmod{\cA} \rightarrow \rmod{\cP}.
$$
Moreover, the composition
$$
\cA \rightarrow \rmod{\cA} \rightarrow \rmod{\cP}
$$
of this restriction with the Yoneda embedding is fully faithful.
\end{lemma}
\begin{proof}
Let $Q \rightarrow P \rightarrow A$ be a presentation of $A \in \cA$ by objects in $\cP$.
Since $\cP$ consists of projective objects, the restricted sequence of functors
    \begin{center}
  \begin{tikzpicture}[mylabel/.style={fill=white}]
      \coordinate (r) at (3,0);
      
      \node (A) {$\cP(-,Q)$};
      \node (B) at ($(A)+(r)$) {$\cP(-,P)$};
      \node (C) at ($(B) + (r)$) {$\cA(-,A)|_{\cP}$};
      \node (D) at ($(C) + (r)$) {$0$};
      
      \draw[->,thick] (A) to (B);
      \draw[->,thick] (B) to (C);
      \draw[->,thick] (C) to (D);
  \end{tikzpicture}
\end{center}
is exact, and hence $\cA(-,A)|_{\cP}$ is finitely presented.
It follows that we have a well-defined functor $(A \mapsto \cA(-,A)|_{\cP}): \cA \rightarrow \rmod{\cP}$. By the universal property of $\rmod{\cA}$, this functor extends to a right exact functor $\rmod{\cA} \rightarrow \rmod{\cP}$ which is isomorphic to the desired restriction functor.
Fully faithfulness of its composition with the Yoneda embedding easily follows from \Cref{remark:commutative_square_calculus}.
\end{proof}

\begin{definition}\label{definition:weak_cokernel}
    A \emph{weak kernel} of a morphism $A \xrightarrow{\alpha } B$ in $\cA$ is given by a morphism $K \xrightarrow{\kappa} A$ in $\cA$ such that
\begin{center}
  \begin{tikzpicture}[mylabel/.style={fill=white}]
      \coordinate (r) at (4,0);
      
      \node (A) {$\cA(-,K)$};
      \node (B) at ($(A)+(r)$) {$\cA(-,A)$};
      \node (C) at ($(B) +(r)$) {$\cA(-,B)$};
      
      \draw[->,thick] (A) to node[above]{$\cA(-,\kappa)$} (B);
      \draw[->,thick] (B) to node[above]{$\cA(-,\alpha)$}(C);
  \end{tikzpicture}
\end{center}
is an exact sequence of functors, and we call $K \xrightarrow{\kappa} A \xrightarrow{\alpha } B$ a \emph{weak kernel sequence}. Dually, a \emph{weak cokernel} of $A \xrightarrow{\alpha } B$ is given by a morphism $B \xrightarrow{\epsilon} C$ such that
\begin{center}
  \begin{tikzpicture}[mylabel/.style={fill=white}]
      \coordinate (r) at (4,0);
      
      \node (A) {$\cA(A,-)$};
      \node (B) at ($(A)+(r)$) {$\cA(B,-)$};
      \node (C) at ($(B) +(r)$) {$\cA(C,-)$};
      
      \draw[<-,thick] (A) to node[above]{$\cA(-,\alpha)$} (B);
      \draw[<-,thick] (B) to node[above]{$\cA(-,\epsilon)$}(C);
  \end{tikzpicture}
\end{center}
is an exact sequence of functors, and we call $A \xrightarrow{\alpha } B \xrightarrow{\epsilon} C$ a \emph{weak cokernel sequence}.
\end{definition}

\begin{theorem}\label{theorem:Freyd_main}
The category $\rmod{\cA}$ is abelian if and only if $\cA$ has weak kernels.
In that case, kernels are given pointwise.
Moreover, a right exact functor $F: \rmod{\cA} \rightarrow \cB$ into an abelian category $\cB$ is exact if and only if the restricted functor $\cA \rightarrow \rmod{\cA} \rightarrow \cB$ maps weak kernel sequences to exact sequences.
\end{theorem}
\begin{proof}
This is due to Freyd \cite{Freyd66}, see also \cite{Krause98}*{Section 2}.
\end{proof}

\bibliography{bib}
\bibliographystyle{amsrefs}%

\end{document}